\documentclass[a4paper]{article}
\usepackage[all]{xy}\usepackage[latin1]{inputenc}        
\usepackage[dvips]{graphics,graphicx}
\usepackage{amsfonts,amssymb,amsmath,color,mathrsfs, amstext}
\usepackage{amsbsy, amsopn, amscd, amsxtra, amsthm,authblk}
\usepackage{enumerate,algorithmic,algorithm}
\usepackage{upref}
\usepackage{geometry}
\usepackage[displaymath]{lineno}
\usepackage{float}
\usepackage{yhmath}
\usepackage{booktabs}
\usepackage{subcaption}
\usepackage{multirow}
\usepackage{makecell}
\usepackage[colorlinks,
            linkcolor=red,
            anchorcolor=red,
            citecolor=red
            ]{hyperref}

\numberwithin{equation}{section}  

\def\b{\boldsymbol}

\def\e{\epsilon}

\def\R{\mathbb{R}}
\def\P{\mathbb{P}}
\def\E{\mathbb{E}}
\def\cL{\mathcal{L}}

\def\cI{\mathcal{I}}

\def\cX{\mathcal{X}}
\def\cI{\mathcal{I}}
\def\bfP{\mathbf{P}}

\newtheorem{theorem}{Theorem}[section]
\newtheorem{lemma}{Lemma}[section]
\newtheorem{proposition}{Proposition}[section]
\newtheorem{remark}{Remark}[section]
\newtheorem{corollary}{Corollary}[section]
\newtheorem{definition}{Definition}[section]
\newtheorem{assumption}{Assumption}[section]

\makeatletter
\@mparswitchfalse
\makeatother

\title{A collision-oriented interacting particle system for Landau-type equations and the molecular chaos}

\author[a,c]{
Kai Du \thanks{Email: kdu@fudan.edu.cn}
}
\author[b,c]{
Lei Li \thanks{Email: leili2010@sjtu.edu.cn}
}

\affil[a]{Shanghai Center for Mathematical Sciences, Fudan University, Shanghai 200438, China
}
\affil[b]{School of Mathematical Sciences, Institute of Natural Sciences, MOE-LSC, Shanghai Jiao Tong University, Shanghai 200240, China}
\affil[c]{Shanghai Artificial Intelligence Laboratory}

\date{}

\begin{document}

\maketitle

\begin{abstract}
We propose a collision-oriented particle system to approximate a class of Landau-type equations. This particle system is formally derived from a particle system with random collisions in the grazing regime, and happens to be a special random batch system with random interaction in the diffusion coefficient. The difference from usual random batch systems with random interaction in the drift is that the batch size has to be $p=2$. We then analyze the convergence rate of the proposed particle system to the Landau-type equations using the tool of relative entropy, assuming that the interaction kernels are regular enough. A key aspect of our approach is the gradient estimates of logarithmic densities, applied to both the Landau-type equations and the particle systems. Compared to existing particle systems for the approximation of Landau-type equations, our proposed system not only offers a more intrinsic reflection of the underlying physics but also reduces the computational cost to $O(N)$ per time step when implemented numerically.
\end{abstract}

\section{Introduction}

We consider in this work the Landau-type equations:
\begin{multline}\label{eq:mckeanlandau}
\partial_t f(t,v)=-\nabla_v\cdot\left( \int_{\R^d} K(v-v_*)f(t, v_*)\,d v_* f(t, v)\right) \\
+\frac{1}{2}\nabla_v^2:\left(\int_{\R^d} A(v-v_*)f(t, v_*)dv_* f(t, v) \right).
\end{multline}
Here, $f$ is the particle density of particles in the velocity space $\R^d$ and the dimension $d\ge 2$. 
The kernel $K$ is the interaction in the drift and $A(\cdot)$ is the interaction kernel for the diffusion coefficient. 
We note that both $K$ and $A$ arise from the collision in the Landau kinetic equation
\begin{gather}
\partial_t f=Q_L(f,f)
=\frac{1}{2}\nabla_v\cdot(A*f\nabla_v f-f A* \nabla_v f),
\end{gather}
where $Q_L$ models the collision in the so-called grazing regime (see Section~\ref{sec:derivation} for more details). The notation
\begin{gather}
A*f(v):=\int_{\R^d} A(v-v')f(v')\,dv'
\end{gather}
represents the convolution in $\R^d$.
It can be verified easily that $K(z)=\nabla\cdot A(z)$ (note that $A$ is symmetric) in the Landau kinetic equation. In this work, we will consider general models without assuming $K=\nabla\cdot A$ and thus call \eqref{eq:mckeanlandau} ``Landau-type'' models. 
We call $K$ the ``drift'' and $A$ the ``diffusion''  just for the convenience.

The Landau kinetic equation is an important model in kinetic theory \cite{villani2002review}, and can be derived formally from the Boltzmann equation and other simplified Boltzmann like models in the grazing regime \cite{villani1998new,guerin2004pointwise,desvillettes1992asymptotics}.
In the Landau equation, the kernel $A$ is given by
\begin{gather}
A(z)=\Lambda |z|^{\gamma+2}(I- \Pi(z))=\Lambda |z|^{\gamma+2}(I-\hat{z}\hat{z}),\quad K=-\Lambda z,
\end{gather}
where $\Lambda$ is a constant related to the collision kernel, $-d\le \gamma\le 1$ and $\hat{z}=z/|z|$. 
If $d=3$, the cases $\gamma\in (0, 1]$ correspond to the hard potentials and $\gamma=0$ corresponds to the Maxwellian potential, while the cases $\gamma\in [-3, 0)$ correspond to the soft potentials. In particular, $\gamma=-d=-3$ corresponds to the Coulomb potential. The well-posedness of the Landau equations for various cases has been jusitfied in a series of works \cite{desvillettes2000spatially,fournier2021stability,villani1998spatially,fournier2009well,guillen2023landau}.

The derivation of Landau equation from interacting particle systems is considered important and challenging in mathematical physics. Moreover, developing suitable particle methods for solving Landau models numerically is also of great significance in plasma physics. Guerin et al. obtained some early works on particle approximations for Landau type models \cite{guerin2003convergence,fontbona2009measurability}. Later, there have been various of results on the particle approximation of Landau equations, including the Maxwellian case \cite{carrapatoso2016propagation}, hard potential cases \cite{fournier2017kac}, and soft potential cases \cite{fournier2016propagation}. These results are all based on the coupling technique. Recently, Carrillo et al. obtained a propagation of chaos results for a Landau-type model using the relative entropy \cite{carrillo2024mean} and for the Maxwellian case \cite{carrillo2024entropy}.  Carrillo et al. also designed a series of numerical particle methods for the Landau equations, see \cite{carrillo2020particle, carrillo2021random,carrillo2023convergence}.

We are interested in proposing a new particle approximation of the Landau-type models \eqref{eq:mckeanlandau} inspired by the microscopic pairwise collisions. The particle system obtained is quite different from those in \cite{fontbona2009measurability,fournier2016propagation,fournier2017kac,carrillo2024mean,carrillo2024entropy}, where the $N$-particle interaction was used in the diffusion. 
In this work, we will consider the following collision-oriented particle system~\ref{alg:rbmparticle} shown in the diagram, to approximate the Landau type equation \eqref{eq:mckeanlandau}. In particular, with a time step $\Delta t$ and time points $t_m=m\Delta t$,  the particles are divided into $N/2$ subgroups at each time point $t_m$, and then they interact within each group for time $\Delta t$ to approximate the statistical effect of the collision. The diffusion for each small time interval is thus only pairwise interaction instead of $N$-particle interaction. The time step $\Delta t$ reflects the frequency of the jumps in the random Poisson process. Clearly, if our particle system is used as a numerical method, the computational cost is only $O(N)$ per time step, which may be a benefit in applications and this would be studied in more detail in the future.  

\begin{algorithm}[H]
\floatname{algorithm}{Particle system}
\caption{Collision-oritented particle system for Landau-type equations}\label{alg:rbmparticle}
\begin{algorithmic}[1]
\STATE Take a time step $\Delta t$ for the mean time duration between two consecutive collisions.
Define $t_m=m\Delta t$.
\STATE At time $t_m$, the particles are divided into $N/2$ subgroups randomly. 
\STATE Consider the time interval $[t_m, t_{m+1})$, the particles evolve as following.
\begin{gather}
dV_i=K(V_i-V_{\theta(i)})\,dt+\sigma(V_i-V_{\theta(i)})dW_i,
\end{gather}
where
\begin{gather}
\sigma(z)\sigma^T(z)=A(z).
\end{gather}
and  $\theta(i)$ is the batchmate of $i$ in the random grouping at $t_m$. If $A=\Lambda |z|^{\gamma+2}(I- \Pi(z))$, one can simply take 
\begin{gather}
\sigma(z)=\sqrt{\Lambda}|z|^{\gamma/2+1}\left(I-\Pi(z) \right).
\end{gather}
\end{algorithmic}
\end{algorithm}

Here, the main difference is that the diffusion is pairwise for each small duration of time interval and then the interaction will be switched. The details for why this particle system is a reasonable approximation will be provided in Section~\ref{sec:derivation}.
Roughly speaking, the statistical effect of the collision between two particles in the grazing regime would be similar to a certain diffusion after a time step. Though the collisions happen at random time points which can be described by some Poisson process, the statistical properties could be approximated by 
interaction that happens at some given time points. 
We remark that the interacting particle system~\ref{alg:rbmparticle} is a special random batch system introduced in \cite{jin2020random,jin2022random,jin2021convergence,jin2022} with $p=2$ applied in the diffusion term. 
This suggests that the random batch interaction has a physical meaning in the Landau type models, not just a numerical approximation. The random batch method was introduced as a numerical method in \cite{jin2020random,jin2021convergence}, where the the particle interaction in the drift term was approximated by the random batch interaction. The value of the batch size $p$ is not very crucial in those models and it has been proved that for any batch size $p\ge 2$, the error goes to zero as $\Delta t\to 0$. We emphasize, however, that the batch size $p=2$ is very crucial for our model as the interaction appears in the diffusion term. In fact,  if one chooses $p\ge 3$ for the random batch approximation in the diffusion, it would not be a consistent approximation for the Landau type equations. 

To see the intuition, let us first fix the groups during $(t_m, t_{m+1}]$ and assume that the particles are independently and identically distributed (i.i.d.) at $t_m$ (there is particle chaos). Then, the law of a specified particle $i$ satisfies for $t\in (t_m, t_{m+1}]$ the following
\begin{gather}\label{eq:joint1}
\partial_t\tilde{f}_i^{\b{\theta}}= -\nabla_{v_i}\cdot \big[K(v_i-v_{\theta(i)}) \tilde{f}_i^{\b{\theta}} \big]+\frac{1}{2}\nabla_{v_i}^2:\big[\sigma\sigma^T(v_i-v_{\theta(i)})\tilde{f}_i^{\b{\theta}}\big].
\end{gather}
With the assumption that the particles are i.i.d. at $t_m$, the correlation would be very weak during $(t_m, t_{m+1}]$. Taking expectation over the random batches $\b{\theta}$ and the other particles, one finds that $\tilde{f}=\E_{\b{\theta}}\tilde{f}_i^{\b{\theta}}$ (we omitted the subindex $i$ due to the symmetry) satisfies
\begin{gather}\label{eq:aprpoximateeqn}
\partial_t\tilde{f}\sim -\nabla_{v_i}\cdot \big[(K(v_i)*\tilde{f})\tilde{f} \big]+\frac{1}{2}\nabla_{i}^2:\big[(\Lambda*f)(v_i)\tilde{f} \big],
\quad \Lambda=\sigma\sigma^T.
\end{gather}
This is the desired Landau type equation.  However, if we apply the random batch approximation with $p\ge 3$ (namely, the group $\theta(i)$ that contains $i$ has $p\ge 3$ particles), the approximation is not correct. To see this, we focus on the diffusion term only and the dynamics could be given by
\begin{gather}
dV_i=\frac{1}{p-1}\sum_{j\in \theta(i), j\neq i}\sigma(V_i-V_{j})dW_i.
\end{gather}
Then, the equation during the short time interval would be
\begin{gather*}
\partial_t\tilde{f}_i^{\b{\theta}}= \frac{1}{2}\nabla_{v_i}^2:\bigg[\Big(\frac{1}{p-1}\sum_{j\in \theta(i), j\neq i}\sigma(v_i-v_j)\Big)\Big(\frac{1}{p-1}\sum_{j\in \theta(i), j\neq i}\sigma(v_i-v_j)\Big)^T\tilde{f}_i^{\b{\theta}}\bigg].
\end{gather*}
If we take the expectation, we cannot obtain the desired form of Landau type equation due to the nonlinearity. 
This effect of the interplay between diffusion and random batch approximation has been investigated in  \cite{guo2024convergence} already, where the McKean equation of the following type has been studied
\begin{gather}
dX=K*\rho(X)\,dt+\sigma*\rho(X)\, dW.
\end{gather}
Here $\rho=\mathcal{L}(X)$ is the law of $X$, and the corresponding equation for the law reads
\begin{gather}
\partial_t \rho=-\nabla\cdot( (K*\rho)(x) \rho(t, x))
+\frac{1}{2}\nabla^2:((\sigma*\rho)(\sigma*\rho)^T \rho(t, x)).
\end{gather}
To approximate this type of equation, one needs to take the batch size $p$ to be big enough ($p\ge \Delta t^{-\beta}$ for some $\beta>0$) \cite{guo2024convergence}.

In this paper, we aim to study the approximation error of the new particle system~\ref{alg:rbmparticle} to Landau type equations in a rigorous framework, assuming the drifts and diffusion terms possess some good properties. The tool we adopt is the relative entropy, which has been proved powerful to study the mean-field limit and propagation of chaos \cite{jabinquantitative,jabin2016mean,huang2024mean}. On one side, we aim to establish the propagation of chaos of the system so that it would be helpful to understand the molecular chaos assumption in the derivation of Landau equations. On the other side, we will also establish the discretization error in terms of the time step $\Delta t$. 
The work here could bring new understanding to the microscopic models of Landau type kinetic models, and inspire new numerical methods.

The rest of the paper is organized as follows. In Section~\ref{sec:derivation}, we explain in detail how the particle system we consider is proposed based on the collision model in the grazing regime. From the derivation, we will see clearly that the random batch interaction in the diffusion term in this collision-oriented particle system is not constructed due to the numerical approximation. Instead, it is connected closely to the underlying physical significance. 
In Section~\ref{sec:mainresults}, we present our main results about the estimates for the approximation error of the particle system to the Landau type equations, and then make some discussions about the molecular chaos. 
Section~\ref{sec:meanfieldproperty} and Section~\ref{sec:Npardensity} are devoted to some key estimates needed for the proof. In particular, we establish a Li--Yau type gradient estimate for the mean-field Landau type equation in Section~\ref{sec:meanfieldproperty} and an integral version of the estimates for the derivatives of the logarithmic density of the particle system in Section~\ref{sec:Npardensity}.  We finish the proof by closing some loose ends in Section~\ref{sec:looseends}
and conclude the work in Section~\ref{sec:conclusion}.

\section{A formal derivation for the collision-oriented particle system}\label{sec:derivation}

In this section, we illustrate why the particle system~\ref{alg:rbmparticle} considered in this paper is meaningful.  In particular, we check the grazing limit in the particle level of the Boltzmann type collision operator for the Maxwell case as discussed in \cite{fournier2000existence,
guerin2004pointwise}. In the grazing regime, the leading terms naturally yield diffusion operator considered in this paper and thus justifies our model.

\subsection{A particle system in the velocity space for Kac equation}

Consider the Kac equation in 2D, which is simplified spatially homogeneous Boltzmann equation, for Maxwellian particles 
\begin{gather}\label{eq:kacequation}
\begin{split}
\partial_t f& =Q(f, f)=\int_{v_*\in \R^2}
\int_{-\pi}^{\pi}[f(t, v')f(t, v_*')-f(t, v)f(t, v_*)]\beta(\theta)d\theta dv_* \\
&=:Q_g(f,f)-Q_{\ell}(f,f).
\end{split}
\end{gather}
Here, $\beta$ is the collision kernel and only depends on $\theta$, which is the angle of rotation due to collision. Moreover, it is an even function of $\theta$
\begin{gather}
\beta(\theta)=\beta(-\theta).
\end{gather}

A particle with velocity $v$ collides with a particle with velocity $v_*$,  and $v$ is then changed into $v'$ while $v_*$ is changed to $v_*'$ after collision.
For elastic collision, both the total momentum and kinetic energy are conserved. Specifically, given the angle change of the relative velocity $\theta$, one has 
\begin{gather}\label{eq:velocityaftercollision}
\begin{split}
& v'=\frac{1}{2}(v+v_*)+R(\theta)\frac{v-v_*}{2}
=v+\frac{1}{2}(R(\theta)-I)(v-v_*)=:c(v, v_*,\theta),\\
& v_*'=\frac{1}{2}(v+v_*)-R(\theta)\frac{v-v_*}{2}
=v_*-\frac{1}{2}(R(\theta)-I)(v-v_*)=:c_*(v, v_*, \theta).
\end{split}
\end{gather}
The matrix $R(\theta)$ is the standard rotation matrix
\begin{gather}
R(\theta)
=\begin{pmatrix}
\cos\theta & -\sin\theta\\
\sin\theta &\cos\theta
\end{pmatrix},
\end{gather}
 and one has
\begin{gather}
v'-v_*'=R(\theta)(v-v_*).
\end{gather}
The loss of density due to collision at $v$ is thus given by 
\begin{gather}
\bar{Q}_{\ell}(v)=\iint_{\theta, v_*} f^{(2)}(t, v, v_*)\beta(\theta) d\theta dv_* ,
\end{gather}
where $f^{(2)}$ is the two marginal distribution of the particles and the overall loss is thus obtained by integrating over all possible $v_*$ and $\theta$. One often assumes the molecular chaos, which means the two-marginal distribution to be written as the product of one marginals: 
\begin{gather}\label{eq:molecularchaos}
f^{(2)}(t, v_1, v_2)= f^{(1)}(t, v_1)f^{(1)}(t, v_2).
\end{gather}
Then, $\bar{Q}_{\ell}$ reduces to $Q_{\ell}$, the second term on the right-hand side of \eqref{eq:kacequation}.
Reversely, when particles with velocity $v'$ and $v_*'$ given in \eqref{eq:velocityaftercollision} collide and the angle is $-\theta$, the resulted velocities will be $v$ and $v_*$. The gain of the density is thus given by
\begin{gather}
\begin{split}
\bar{Q}_g(v)&=\iiint f^{(2)}(t, v', v_*')\beta(-\theta)\delta(v-c(v', v_*', -\theta)) d\theta dv' dv_*'\\
&=\iint_{\theta, v_*} f^{(2)}(t, v', v_*') \beta(\theta) dv_*.
\end{split}
\end{gather}
The first equality is due to the physical meaning of $\bar{Q}_g$ as one should consider every possibility of $v'$ and $v_*'$ that resulted in $c(v', v_*', -\theta)=v$. The coefficient of $\delta(v-c(v', v_*', -\theta))$ is determined by the fact that the integral of $\bar{Q}_g(v)$ over $v$
should be the total gain given by $\iiint f^{(2)}(t, v', v_*')\beta(-\theta)d\theta dv' dv_*'$. For the second equality,  one does the change of variables $w=c(v', v_*', -\theta), w_*=c_*(v', v_*', -\theta)$ and the Jacobian for $dv'dv_*'\to dw dw_*$ is $1$. 
Then, change the integral variable $w_*$ to $v_*$ gives the second equality. One again, assuming \eqref{eq:molecularchaos}, one obtains $Q_g$.

Due to the interpretation of the gain and loss terms, it is easily to  obtain a particle system in the velocity space for the Kac equation, similar to the classical Monte Carlo methods for kinetic equations \cite{guerin2003convergence,bird1976molecular}. The Direct Simulation Monte Carlo (DSMC) for Boltzmann equations in \cite{bird1976molecular} has indeed gained wide applications.
\begin{algorithm}[H]
\floatname{algorithm}{Particle system}
\caption{Collision model for Kac's equation}\label{alg:particlekac}
\begin{algorithmic}[1]
\STATE Take $N$ particles (in velocity space) with labels from $f_0$.
\STATE Generate a sequence of Poisson jumps with strength $N$.
\STATE At each Poisson jump, choose a pair of particles with labels $v$ and $v_*$. Choose $\theta$ from distribution $\beta(\cdot)$. The new labels of the two particles are given by \eqref{eq:velocityaftercollision} or \eqref{eq:velocityaftercollision2}.
\end{algorithmic}
\end{algorithm}

Here, the strength of the Poisson jumps is $N$ so that there would be $O(N)$ collisions during a time interval of length $O(1)$. That means, an observed particle is likely to be involved in one collision.

Note that this particle system is not totally equivalent to the Kac equation because the particles are not independent for finite $N$.  In fact, the equivalence of the particle system to the Landau equation does not hold if $N$ is not big because the molecular chaos \eqref{eq:molecularchaos}  does not hold. However, one expects it to hold as $N\to\infty$. The rigorous justification of chaos is in fact very important and the central point for the mean field limit \cite{jabin2014}.

\subsection{The grazing regime and the particle systems}

In the grazing regime, the small angle collisions are more important. Consider that the collision kernel rewritten as
\begin{gather}
\beta^{\e}(\theta)\sim \frac{1}{\e}g\left(\frac{\theta}{\e} \right),
\end{gather}
where $\e\ll 1$ and $g(\cdot)$ is a function that decays fast enough. Here, we use the notation $\sim$ to indicate that $g$ is actually is the limit of $\e \beta^{\e}(\e z)$ as $\e\to 0$.
We take a time scale as $1/\e^2$ so that $\tilde{t}=\e^2 t$. Using the scaled time $\tilde{t}$, a typical collision happens under a Poisson process with strength $N/(2\e^2)$: we have $N$ particles in the velocity space, and with a Poisson clock with strength $N/(2\e^2)$, two particles are chosen and then collide. 
 See the process illustrated in \ref{alg:particlegrazing} for the particle system with this grazing scaling.  Intuitively, this means that a typical collision happens every $2\e^2/N$ time, and during a time step $\Delta \tilde{t}=\epsilon^2$, typically $N/2$ collisions would happen. 
\begin{remark}
Note that we do not use the rescaled kernel $\frac{1}{\e^3}g(\frac{\theta}{\e})$ as in \cite{fournier2000existence,guerin2004pointwise} because we would like the integral of this rescaled kernel is a constant (i.e., $1$) so that it can be interpretated as a probability density. The extra $\e^2$ will naturally be viewed as the time step, and thus a rescaling of time using $\e^2$ will recover the classical Landau operator.
\end{remark}

\begin{algorithm}[H]
\floatname{algorithm}{Particle system}
\caption{Collision model for Kac's equation with grazing scalings}\label{alg:particlegrazing}
\begin{algorithmic}[1]
\STATE Take $N$ particles (in velocity space) with labels from $f_0$.
\STATE Generate a sequence of Poisson jumps with strength $N/(2\e^2)$.
\STATE At each Poisson jump, choose a pair of particles with labels $v$ and $v_*$. Choose $\theta$ from distribution $\beta^{\e}(\theta)=\e^{-1}g(\e^{-1}\theta)$. The new labels of the two particles are given by \eqref{eq:velocityaftercollision}.
\end{algorithmic}
\end{algorithm}

In this grazing regime, consider $v$ and $v_*$ are given. The velocity is changed from $v$ to $v'$ with $\theta$ being drawn from $\beta^{\e}$ randomly. The mean drift is thus given by
\begin{gather}
d(v, v_*):=\langle v'-v\rangle_{\theta}=\int_{-\pi}^{\pi} \frac{1}{2}(R(\theta)-I)(v-v_*)\beta^{\e}(\theta)f(t, v_*)d\theta.
\end{gather}
Note that $\beta^{\e}(\cdot)$ is an even function, and then one finds that
\begin{gather}
d(v, v_*)=\left(-\frac{1}{2}\int_{-\pi}^{\pi}(1-\cos\theta)\beta^{\e}(\theta)d\theta\right)(v-v_*),
\end{gather}
and
\begin{gather}
\frac{1}{2}\int_{-\pi}^{\pi}(1-\cos\theta)\beta^{\e}(\theta)d\theta
=\frac{\epsilon^2}{2}\int_0^{\infty}z^2 g(z)dz+O(\epsilon^4).
\end{gather}
Hence, 
\begin{gather}\label{eq:meanvelocity}
d(v, v_*)=-\Lambda (v-v_*)\epsilon^2+O(\epsilon^4),
\end{gather}
Here, $\Lambda=\frac{1}{2}\int_0^{\infty}z^2 g(z)dz$.

The second moment for the velocity change is similarly given by
\begin{gather}
\begin{split}
\langle (v'-v) \otimes (v'-v)\rangle_{\theta}
&=\frac{1}{4}\int [(R(\theta)-I)(v-v_*)]^{\otimes 2}\beta^{\e}(\theta)\,d\theta\\
&=\frac{1}{4}\e^2 [J(v-v_*)]^{\otimes 2}\int_{-\infty}^{\infty} z^2 g(z)\,dz
+O(\e^4),
\end{split}
\end{gather}
where
\begin{gather}
J=\begin{pmatrix} 
0 & -1 \\
1 & 0
\end{pmatrix}.
\end{gather}
Since $(J(v-v_*))^{\otimes 2}=|v-v_*|^2(I-\Pi)$ with $\Pi=\frac{(v-v_*)^{\otimes 2}}{|v-v_*|^2}$, one thus has
\begin{gather}
\langle (v'-v) \otimes (v'-v)\rangle=\e^2|v-v_*|^2(I-\Pi)\Lambda
+O(\e^4)=\e^2 A(v-v_*)+O(\e^4),
\end{gather}
where
\begin{gather}\label{eq:CoefficientA}
A(z):=\Lambda |z|^2(I-\Pi(z)),\quad \Pi(z)=\hat{z}\hat{z}.
\end{gather}
From this perspective, due to the randomness of $\theta$ sampled from $\beta^{\e}(\theta)$, the collision behaves in the leading term is like a Markov process with generator
\[
\e^2(-\Lambda (v-v_*)\cdot\nabla+\frac{1}{2}\Lambda |v-v_*|^2(I-\Pi(v-v_*)):\nabla^2)+O(\e^4).
\]

\begin{remark}
The above computation is often argued using the weak formulation for Kac equation. Taking a test function $\phi(v)$, one has
\begin{gather}\label{eq:Qweak}
\langle Q(f, f), \phi\rangle
=\int_{v\in \R^2}\int_{v_*\in \R^2}\int_{-\pi}^{\pi}f(t, v)f(t, v_*)[\phi(v')-\phi(v)]d\theta  dv_* dv
\end{gather}
This can be obtained by the evenness of $\beta^{\e}$ and change of variables
\begin{gather*}
\int_v\int_{v_*}
\int_{-\pi}^{\pi}f(t, v')f(t, v_*')\beta^{\e}(\theta)\phi(v)d\theta dv_*dv
=\int_{v}\int_{v_*}\int_{-\pi}^{\pi}f(t, v)f(t, v_*)\phi(v')\beta^{\e}(\theta)d\theta dv_*dv .
\end{gather*}
The weak formulation \eqref{eq:Qweak} clearly reflects the fact that $v$ is turned into velocity $v'$ after collision. One can decompose to split the mean drift part to obtain
\[
\phi(v')-\phi(v)=[\phi(v')-\phi(v)-(v'-v)\cdot\nabla\phi(v)]
-\nabla\phi(v)\cdot \frac{1}{2}(I-R(\theta))\cdot(v-v_*),
\]
In the grazing regime, $\beta(\theta)\to \beta^{\e}(\theta)=\frac{1}{\e}\beta(\theta/\e)$. One can do Taylor expansion and repeat the computation above to find that
\[
\langle Q(f,f), \phi\rangle=\e^2\iint_{\R^2\times\R^2}(-\Lambda (v-v_*)\cdot\nabla\phi+\frac{1}{2}\Lambda |z|^2(I-\hat{z}\hat{z}):\nabla^2\phi)f(t, v)f(t, v_*)dv dv_*+O(\e^4).
\]
\end{remark}

The coefficient in \eqref{eq:CoefficientA} is degenerate. For technical reasons, we will also consider coefficients of the form
\begin{gather}
A=\Lambda |z|^2(I-\Pi(z))+\sigma^2.
\end{gather}
This corresponds to collisions with some injected noise:
\begin{gather}\label{eq:velocityaftercollision2}
\begin{split}
& v'=v+\frac{1}{2}(R(\theta)-I)(v-v_*)+\xi,\\
& v_*'=v_*-\frac{1}{2}(R(\theta)-I)(v-v_*)+\xi_*.
\end{split}
\end{gather}
The injected noise $\xi$ could be understood as the momentum input due to random environment. For example, in the plasma simulations, there may be photons exchange between the particles and the environment. Or, if we consider ions in the continuum background of electrons. 
We assume that $\E \xi=0$ and
\begin{gather}
\E(|\xi|^2)=\e^2 \sigma^2+O(\e^4).
\end{gather}
for the grazing regime. With the same calculation, due to the mean zero assumption on $\xi$,  it is easy to find that
$d(v, v_*)$ is given by the same  expression \eqref{eq:meanvelocity}, and
\begin{gather}
\begin{split}
\langle (v'-v) \otimes (v'-v)\rangle_{\theta}
&=\frac{1}{4}\e^2 [J(v-v_*)]^{\otimes 2}\int_{-\pi}^{\pi} z^2 g(z)\,dz
+\E(|\xi|^2)+O(\e^4),
\end{split}
\end{gather}
This gives the desired $A$ and the relation $K=\nabla\cdot A$ still holds.

In the grazing regime,  the collision is like the discretization of a diffusion operator with time step $\e^2$. In this sense, the scaled time $\tilde{t}$ is suitable and we take $\Delta \tilde{t}=\epsilon^2$ (correspondingly $\Delta t=1$). For such a time step, we can thus approximate the collision with the Euler--Maruyama discretization of a diffusion process, and further with the time continuous diffusion process with time length $\Delta \tilde{t}=\e^2$.  Due to the strength of the Poisson process, during time step $\Delta \tilde{t}$, any focused particle is likely to be involved in one collision and typically $N/2$ collisions would happen. 
To make the particle simulation more efficient, we take an ideal approximation: for each duration of time $\Delta t=\e^2$, we divide the particles into $N/2$ groups. They evolve according to the diffusion process within each group. That means we use the particle system indicated in Particle system~\ref{alg:rbmlandau}.
\begin{algorithm}[H]
\floatname{algorithm}{Particle system}
\caption{Collision oriented particle system for Landau equation}\label{alg:rbmlandau}
\begin{algorithmic}[1]
\STATE Every time step $\Delta \tilde{t}=\e^2$. The $N$ particles are divided into $N/2$ groups randomly.
\STATE With each group, two particles evolve with the SDE for time step $\Delta t=\e^2$:
\[
dV_i=-\Lambda(V_i-V_{\theta(i)})\,dt+\sqrt{A(V_i-V_{\theta(i)})}\,dW_i.
\]
\end{algorithmic}
\end{algorithm}
If we relax the constraint $K=\nabla\cdot A$ and allow general $K, A$, we then have particle system~\ref{alg:rbmparticle}. As we mentioned, this particle system is exactly a special random batch particle system studied in \cite{jin2020random}, in the diffusion coefficient.
This indicates that the random batch system is not just a numerical algorithm and it has some physical significance in some problems.

\begin{remark}
Here, the Brownian motions $W_i$ are assumed independent.  During the time interval $(t_m, t_{m+1})$, particle $i$ and its batchmate clearly are related due to the conservation of momentum. In this sense, one should have $W_{\theta(i)}=-W_i$ for $t\in (t_m, t_{m+1})$. However, if the particle number is large, this correlation is expected to have little effect.
\end{remark}

As commented above already, the Kac's equation needs the chaos assumption that $f^{(2)}(t, v_1, v_2)\approx f^{(1)}(t, v_1)f^{(1)}(t, v_2)$. 
Using the collision model~\ref{alg:particlegrazing}, the molecular chaos is hard to justify directly, and thus the approximation of the particle system to the Landau equation is not clear.
However, using  particle system~\ref{alg:rbmlandau}, there is hope to justify the molecular chaos for regular $A$ because we may use the available tools for propagation of chaos to establish the asymptotic ``molecular chaos'' so that we can establish the convergence of the particle system to the Landau type equations.

We would also like to point out that  the difference of Algorithm \ref{alg:rbmlandau} from Algorithm \ref{alg:particlegrazing} seems not that essential as the numerical results for the two versions of RBM algorithms (RBM with and without replacement) in \cite{jin2020random} suggested. Hence, we believe that if the propagation of chaos can be established for system~\ref{alg:rbmlandau} corresponding to the real Landau model, the eventual molecular chaos result for the original collision model would not be very far.

\section{Approximation error and propagation of chaos}\label{sec:mainresults}

In this section, we present our main results about the approximation error of our particle system~\ref{alg:rbmparticle} to the Landau type equations, and make some discussions. Then, we provide the proof of the main results by taking some important estimates for granted, which will be proved in later sections. 

\subsection{Assumptions, main results and discussions}

Let us first list out our assumptions on $K$ and $A$.
\begin{assumption}\label{ass:KA1}
The kernel $K$ and $A$ are bounded and smooth functions with derivatives up third order to be bounded. Moreover, $A$ is uniformly positive definite in the sense that exists $\kappa>0$ with
\begin{gather}
A(v)\succeq \kappa I_d,\quad \forall v\in \R^d.
\end{gather}
\end{assumption}
The positive definiteness is essential in this work, which does not hold for the real Landau equation without injected noise. However, if there is injected noise, $A\sim |v|^{\gamma+2}(I-\hat{v}\hat{v})+\sigma^2$. The positive definiteness assumption seems reasonable, which is not so artificial.

We will also assume the following.
\begin{assumption}\label{ass:A2}
The matrix $A$ satisfies that
\begin{gather}
\|\nabla_v A(v)\|\le C(1+|v|)^{-1}.
\end{gather}
\end{assumption}
\noindent This assumption is used to ensure that $|\nabla_v\bar{f}|\le C(1+|v|)$ in Theorem~\ref{pro:gradlogdensity}~(ii).
Though this condition is posed as a technical condition, it is satisfied
for the Landau equation with (regularized) $A\sim |v|^{\gamma+2}(I-\hat{v}\hat{v})+\sigma^2$ for $\gamma\le -2$.

We assume the initial velocities are sampled from a common distribution $f_0$ independently. We put the following assumptions on $f_0$.
\begin{assumption}\label{ass:initial}
The initial distribution $f_0$ has a smooth density, denoted by $f_0$ again, with the bounds
\begin{gather}
c_1e^{-\alpha_1 |v|^2}\le f_0(v)\le c_2 e^{-\alpha_2 |v|^2}.
\end{gather}
\end{assumption}

Let $f^{N}\in \bfP_s((\R^{d})^N)$ be the joint probability distribution for the random batch particle system~\ref{alg:rbmparticle} with initial distribution
\[
f^N(0,v_1,\cdots, v_N)=\prod_{i=1}^N f_0(v_i).
\] 
The one marginal distribution is denoted by $f^{(1)}$. Let $\bar{f}$ be the solution to the McKean equation \eqref{eq:mckeanlandau}. Here, we use bar so that the $N$ tensor
\[
\bar{f}^N(t,v_1,\cdots,v_N):=\bar{f}^{\otimes N}(t, v_1,\cdots, v_N):=\prod_{i=1}^N \bar{f}(t, v_i)
\]
can be distinguished from $f^N$. 

Our goal is to show that $f^{(1)}\to \bar{f}$ as $N\to\infty$ in some sense.
Motivated by the tools used for quantative estimate of propagation of chaos in \cite{jabinquantitative,huang2024mean}, we use the tool of relative entropy. 
\begin{definition}
For two measures $\mu, \nu\in \bfP(\cX)$ where $\cX$ is some Polish space, the relative entropy is defined as
\begin{gather}
H(\mu |\nu)=
\begin{cases}
\int_{\cX} \log(\frac{d\mu}{d\nu})\mu(dx), & \mu\ll\nu,\\
+\infty & \text{otherwise}.
\end{cases}
\end{gather}
\end{definition}
In our case, $\cX=(\R^d)^k$ for $k=1,\cdots, N$. Then, one often considers a normalized version of the relative entropy for $\mu,\nu\in \bfP((\R^d)^k)$:
\begin{gather}
H_k(\mu | \nu)=\frac{1}{k}H(\mu | \nu).
\end{gather}
The reason to consider this is that when the particles are close to being independent, then $H(\mu | \nu)$ scales like $k$ for $\mu\approx (\mu^{(1)})^{\otimes k}$ and $\nu=(\nu^{(1)})^{\otimes k}$. A useful property in the analysis of propagation of chaos is the following linear scaling property of entropy.
\begin{lemma}\label{lmm:orderentropy}
For symmetric measures, if $1\le k\le \ell\le N$, one has
\begin{gather}
H_k(f^{(k)}| \bar{f}^{\otimes k})\le H_{\ell}(f^{(\ell)}| \bar{f}^{\otimes \ell}).
\end{gather}
\end{lemma}

The main result in this work is the following.
\begin{theorem}\label{thm:chaosforLandau}
Suppose that Assumptions \ref{ass:KA1}-\ref{ass:A2} for the kernels hold and that the initial density $f_0$ satisfies the Gaussian tail bounds in Assumption \ref{ass:initial}.
Then, it holds that
\begin{gather}
H(f^N| \bar{f}^{\otimes N})\le C(1+N\Delta t).
\end{gather}
Consequently, for any $1\le j\le N$, the distances between $j$-marginal distribution of the particle system and the tensorized law of the Landau-type equation are controlled by
\begin{gather}
TV(f^{(j)}, \bar{f}^{\otimes j})+W_2(f^{(j)}, \bar{f}^{\otimes j})  \le C\sqrt{H(f^{(j)}|\bar{f}^{\otimes j})} \le  C\sqrt{j}\left(\frac{1}{\sqrt{N}}+\sqrt{\Delta t} \right).
\end{gather}
\end{theorem}

We believe that the rate here is not optimal.
As we will discuss in Section~\ref{subsec:localerr}, the optimal rate is expected to be $1+N\Delta t^2$. The main technical difficulty
is the estimate of the integral version of the third order derivatives independent of $N$.

\subsection*{The discussion on the molecular chaos}

As we have discussed in Section~\ref{sec:derivation}, the Kac equation can be shown to converge to the Landau equation. However, the derivation of Kac equation itself (and in fact Boltzmann equation) has assumed the molecular chaos \eqref{eq:molecularchaos}, which has not been justified. 
In fact, justification of the molecular chaos in the collision model \eqref{alg:particlekac} and \eqref{alg:particlegrazing} is very important for the microscopic description of the Landau models.  
As we have explained, the model \eqref{alg:rbmlandau} is expected to be close to \eqref{alg:particlegrazing}. Hence, it is expected that \eqref{alg:particlegrazing} would have the molecular chaos if system \eqref{alg:rbmlandau} does.

The result in Theorem \ref{thm:chaosforLandau} implies the propagation of chaos, because the independence between the particles at $t=0$ is preserved for $t\in (0, T]$ in the $N\to\infty$. It indeed justifies the molecular chaos in the $N\to\infty$ limit. We recall Kac's notation of  chaos (see \cite[Proposition 2.2]{sznitman1991}).
\begin{definition}
Let $E$ be a Polish space and $f$ be a probability measure on $E$. A sequence of symmetric probability measures $f^N$ on $E^N$ is said to be $f$-chaotic, if one of the three following equivalent conditions is satisfied:
\begin{enumerate}[(i)]
\item The sequence of second marginals $f^{(2)} \rightharpoonup f^{\otimes 2}$ as $N\to\infty$;
\item For all $j\ge 1$, the sequence of $j$th marginals $f^{(j)} \rightharpoonup f^{\otimes j}$ as $N\to\infty$;
\item Let $(X^{1N}, \cdots,  X^{N,N}) \in  E^N$ be drawn randomly according to $f^N$. The empirical (random) measure $\mu^N
=\frac{1}{N}\sum_{j=1}^N \delta(\cdot-X^{jN})$ converges in law to the constant probability measure $f$ as $N\to\infty$.
\end{enumerate}
\end{definition}
\noindent Here, the notation $\nu_N\rightharpoonup \nu$ for probability measures in $\mathbf{P}(E)$ means the weak convergence of probability measures (convergence in law), i.e., convergence under the weak star topology against $C_b(E)$.
The result in Theorem \ref{thm:chaosforLandau} clearly means that the joint distribution $f^N$ of our particle system \eqref{alg:rbmlandau} is $\bar{f}$-chaotic. The result tells us that the larger number of particles $N$ is, the more frequent that collision happens ($\Delta t$ smaller), there would be more chaos. 

Though the result here is for regular kernels, establishing the propagation of chaos for the real Landau model could be considered in a similar fashion, at least for some cases (like the hard potential and Maxwellian cases, $\gamma\ge 0$), and then generalized to model \eqref{alg:particlegrazing}. We believe our result is inspiring for the eventual derivation of Landau equations from the particle systems with collisions.

\subsection{Proof of the main results}

In this subsection, we give the proof of the main results while defer some main estimates of the densities to later sections.

Let $\theta^{(m)}:=\{\theta^{(m)}(1), \cdots, \theta^{(m)}(N)\}$ be the random batch at $t_m$. Here, $\theta^{(m)}(i)$ means the batchmate of $i$. Clearly, one has
\[
\theta^{(m)}(\theta^{(m)}(i))=i.
\]
Conditioning on a fixed sequence of batches $\b{\theta}:=(\theta^{(0)}, \theta^{(1)},\cdots)$, the law of the random batch system $(V_1,\cdots, V_N)$, denoted by $\tilde{f}^{\b{\theta}}$, satisfies the following linear Fokker-Planck equation for $t\in (t_m, t_{m+1})$ (and is continuous at $t_{m+1}$):
\begin{equation}\label{eq:jointrbm}
\partial_t\tilde{f}^{\b{\theta}}=\sum_{i=1}^N\left[-\nabla_{v_i}\cdot\left(K(v_i-v_{\theta^{(m)}(i)})\tilde{f}^{\b{\theta}} \right)+\frac{1}{2}\nabla_{v_i}^2:\left(A(v_i-v_{\theta^{(m)}(i)}) \tilde{f}^{\b{\theta}} \right)\right].
\end{equation}
Note that if we observe the random batch particle system at the discrete time $t_m$, it then forms a time-homogeneous Markov chain. By the Markov property, it is not hard to see the law at $t_m$ is given by
\begin{gather}
f^N(t_m, \cdot)=\E_{\b{\theta}} [\tilde{f}^{\b{\theta}}(t_m, \cdot)].
\end{gather}
Now, we consider the following density $f_m^{N,\theta}$, which satisfies
\begin{gather}\label{eq:initialvalueoffm}
f_m^{N,\theta}(t_m,\cdot)=f^N(t_m, \cdot)
\end{gather} 
at $t_m$ but evolves with the batch $\theta^{(m)}$ given. Then, again by the Markov property, one has for $t\in (t_m, t_{m+1})$ that
\begin{gather}
\partial_tf_m^{N,\theta}=\sum_{i=1}^N\left[-\nabla_{v_i}\cdot\left(K(v_i-v_{\theta^{(m)}(i)})f_m^{N,\theta}\right)+\frac{1}{2}\nabla_{v_i}^2:\left(A(v_i-v_{\theta^{(m)}(i)})f_m^{N,\theta}\right)\right],
\end{gather}
and that
\begin{gather}
f^N(t_{m+1},\cdot)=\E_{\theta^{(m)}}f_m^{N,\theta}(t_{m+1}, \cdot).
\end{gather}

Define
\begin{gather}
f^{N}(t,\cdot):=\E_{\theta^{(m)}}f_m^{N,\theta}(t,\cdot)=\E_{\b{\theta}}[\tilde{f}^{\b{\theta}}(t,\cdot)].
\end{gather}
One then has for $t\in (t_m, t_{m+1})$ that:
\begin{gather}
\partial_tf^N=\sum_{i=1}^N\E_{\theta^{(m)}}\left[-\nabla_{v_i}\cdot\left(K(v_i-v_{\theta^{(m)}(i)})f_m^{N,\theta}\right)+\frac{1}{2}\nabla_{v_i}^2:\left(A(v_i-v_{\theta^{(m)}(i)})f_m^{N,\theta}\right)\right].
\end{gather}
Using \eqref{eq:mckeanlandau}, one finds that the equation for $\bar{f}^N$ satisfies the following equation
\begin{gather}
\partial_t\bar{f}^N=\sum_{i=1}^N\left[-\nabla_{v_i}\cdot\left(K(v_i; [\bar{f}])\bar{f}^{N}\right)+\frac{1}{2}\nabla_{v_i}^2:\left(A(v_i; [\bar{f}])f^{N}\right)\right],
\end{gather}
where we introduced the notations $K(v_i; [\bar{f}])$ and $A(v_i; [\bar{f}])$ as the convolved kernel with $\bar{f}$:
\begin{gather}
K(\cdot; [\bar{f}]):=\int K(\cdot-y)\bar{f}(t, y)dy,
\quad A(\cdot; [\bar{f}]):=\int A(\cdot-y)\bar{f}(t, y)dy.
\end{gather}

Our goal is then to establish an inequality of Gr\"onwall type for $H(f^N | \bar{f}^N)$.  To do this, we need some auxilliary results as follows.

The following is a Fenchel--Young type inequality which allows us to change the base probability measure. 
\begin{lemma}\label{lmm:youngIneq}
For two probability measures $\mu,\mu'$ and test function $\Phi$:
\[
\mu(\Phi)\le \frac{1}{\eta}(H(\mu |\mu')+\log\mu'(e^{\eta\Phi})),
\forall \eta>0.
\]
\end{lemma}

Next, we introduce a concentration result for the law of large number type estimates we need later.
Recall the $\psi_2$ norm (the Orlicz norm corresponding to $\psi_2(x)=\exp(x^2)-1$) for sub-Gaussian variable \cite{vershynin2018high}.
\begin{gather*}
\|X\|_{\psi_2}:=\inf\left\{c\ge 0: \E(e^{|X|^2/c^2})\le 2 \right\}.
\end{gather*}
The standard Hoeffding's inequality \cite{vershynin2018high} is a type of concentration inequality, which claims that for $N$ independent centered (i.e., mean-zero) real random variables $\{X_i\}_{i=1}^N$, one has for some universal constant $c_*>0$ that 
\begin{gather}
\P\left( \left|\sum_{i=1}^N X_i \right|\ge y \right) \le 2\exp\left(-\frac{c_* y^2}{\sum_{i=1}^N \|X_i\|_{\psi_2}^2}\right),
\quad \forall y\ge 0.
\end{gather}
In our case, $X_i$ could  be vectors and matrices. We recall a result due to Talagrand \cite[Theorem 4]{talagrand1989isoperimetry}: for $N$ independent mean-zero random variables taking values in a Banach space $B$, one has
\begin{gather}
\|\sum_{i=1}^N X_i\|_{\psi_2}
\le K_2\left(\E \|\sum_{i=1}^N X_i\|_B+(\sum_{i=1}^N \|X_i\|_{\psi_2}^2)^{1/2}\right).
\end{gather}
In our case, the vectors and matrices are actually in $\R^{d'}$ for $d'=d$ or $d'=d^2$. Hence, 
\begin{gather}
\|\sum_{i=1}^N X_i\|_{\psi_2}^2\le 2K_2^2\left(\sum_{i=1}^N (\|X_i\|_2^2+\|X_i\|_{\psi_2}^2)\right)\le K_2'\sum_{i=1}^N \|X_i\|_{\psi_2}^2,
\end{gather}
since $\|X_i\|_2\le \|X_i\|_{\psi_2}$. Hence, the Hoeffding's inequality also holds for our case and the universal constant $c_*$ might be larger compared to the real random variables.

The following is a consequence of the above concentration inequality. The sketch of the proof will be given in Section~\ref{subsec:newlargedevi} for the convenience of the readers. 
\begin{lemma}\label{lmm:concentration}
Consider $\rho\in \mathcal{P}(E)$ and  $\psi(x)$ satisfying 
$\int_{E} \psi(x)\rho(dx) =0$ and for the universal constant $c_*>0$ in the Hoeffding's inequality, the following holds
\begin{gather*}
\|\psi(x)\|_{\rho}:=\inf\left\{c>0: \int_E \exp(|\psi(x)|^2/c^2)\rho(dx) \le 2 \right\}<c_*.
\end{gather*}
Then
\begin{gather}
\sup_{N\geq 1} \int_{E^N} \exp\bigg(\frac{1}{N} \Big|\sum_{i=1}^N \psi(x_i)\Big|^2\bigg)\rho^{\otimes N}\mathrm{d}\mathrm{x}<\infty.
\end{gather}
\end{lemma}


\begin{remark}
The result above is related to the large deviation estimate from \cite[Theorem 4]{jabinquantitative}, proved using combinatoric techniques.
Clearly, one may understand $\phi(x, y)$ from \cite[Theorem 4]{jabinquantitative} by $\phi(x, y)\sim \psi(x)\psi(y)$ in our case. However, for our purpose later, $\sup_y |\psi(x)\psi(y)|=\infty$ so one cannot apply the result from \cite[Theorem 4]{jabinquantitative} directly.
\end{remark}

Next, we introduce a copy of the random batch for our use later.
Given a random batch $\theta$, we consider the following auxiliary coupled random batch $\tilde{\theta}_{i}$.
\begin{itemize}
\item Choose $j$ from $\{1,\cdots, N\}\setminus \{i\}$ randomly with equal probability. Set $\tilde{\theta}_i(i)=j$.
\item If $j=\theta(i)$, then $\tilde{\theta}_i=\theta$. Otherwise, we set $\tilde{\theta}_i(\theta(i))=\theta(j)$
and $\tilde{\theta}_i(\ell)=\theta(\ell)$ for $\ell\neq i, j, \theta(i), \theta(j)$.
\end{itemize}
Regarding this coupled batch, we have the following simple observation.
\begin{lemma}\label{lmm:coupledbatch}
The coupled batch $\tilde{\theta}_i$ itself is a random batch in the sense that it has the same law with $\theta$. Moreover, for any $i$ and $j\neq i$, one has
\begin{gather}
\P(\theta(i)=j | \tilde{\theta}_i)=\frac{1}{N-1}.
\end{gather}
\end{lemma}
\begin{proof}
For the first claim follows by the symmetry of the construction. 

For the second claim, we show that for any concrete permutation
$\sigma$ that can be decomposed into $N/2$ cycles with length 2, one has
\[
\P(\theta(i)=j | \tilde{\theta}_i(k)=\sigma(k))=\frac{1}{N-1}.
\]
In fact, by Bayesian formula, this equals
\[
\frac{\P(\theta(i)=j, \tilde{\theta}_i(k)=\sigma(k),k=1,\cdots, N)}{P(\tilde{\theta}_i(k)=\sigma(k),k=1,\cdots, N)}
=\frac{\P(\theta=\tilde{\sigma}, \tilde{\theta}_i(i)=\sigma(i))}{\P(\tilde{\theta}_i(k)=\sigma(k),k=1,\cdots, N)}.
\]
Here, $\tilde{\sigma}$ is constructed from $\tilde{\theta}_i$ with the information $\theta(i)=j$.
Since $\P(\theta=\tilde{\sigma})=\P(\tilde{\theta}_i=\sigma)$, the claim then follows by the construction of $\tilde{\theta}_i$.
\end{proof}

Now, we complete the proof of the main results taking them for granted, and assuming some local error bounds  in Section~\ref{sec:Npardensity}  being proved. We will close them in later sections.

\begin{proof}[Proof of Theorem \ref{thm:chaosforLandau}]
{\bf Step 1}: Computing the derivative of the relative entropy.

Noting $\int \partial_t f^N=0$, one has
\begin{multline*}
\frac{d}{dt}H(f^N | \bar{f}^N)
=\int (\log\frac{f^N}{\bar{f}^N})\partial_t f^N-\int \frac{f^N}{\bar{f}^N}\partial_t \bar{f}^N\\
=\sum_{i=1}^N\int \nabla_{v_i}\left(\log \frac{f^N}{\bar{f}^N}\right)\cdot \E_{\theta^{(m)}}(
[K(v_i-v_{\theta^{(m)}(i)})-\frac{1}{2}\nabla_{v_i}\cdot A(v_i-v_{\theta^{(m)}(i)})] f_m^{N,\theta} )\\
-\frac{1}{2}\sum_{i=1}^N\int \E_{\theta^{(m)}}\left(\nabla_{v_i}\left(\log \frac{f^N}{\bar{f}^N}\right)
\otimes \nabla_{v_i}f_m^{N,\theta}: A(v_i-v_{\theta^{(m)}(i)})\right)-\\
\sum_{i=1}^N\int \nabla_{v_i}\left(\frac{f^N}{\bar{f}^N}\right)\cdot ([K(v_i; [\bar{f}])-\frac{1}{2}\nabla_{v_i}\cdot A(v_i; [\bar{f}]) ]\bar{f}^N )
+\frac{1}{2}\sum_{i=1}^N\int  \left(\nabla_{v_i}\left( \frac{f^N}{\bar{f}^N}\right)
\otimes \nabla_{v_i}\bar{f}^{N}: A(v_i;\bar{f})\right).
\end{multline*}
Define
\[
b(z):=K(z)-\frac{1}{2}\nabla\cdot A(z),\quad b(v; [\bar{f}])=\int b(v-y)\bar{f}(t, y)dy.
\]
Then, the above can be rearranged into 
\begin{multline*}
\frac{d}{dt}H(f^N | \bar{f}^N)=-\frac{1}{2}\sum_{i=1}^N\int \Bigg( \E_{\theta^{(m)}}\left(\nabla_{v_i}\left(\log \frac{f^N}{\bar{f}^N}\right) \otimes \nabla_{v_i}f_m^{N,\theta}: A(v_i-v_{\theta^{(m)}(i)})\right) \\
- \nabla_{v_i}\left(\log \frac{f^N}{\bar{f}^N}\right)\otimes\nabla_{v_i}\log\bar{f}^N : A(v_i; \bar{f}) f^N \Bigg) +\\
\sum_{i}\int \Bigg( \nabla_{v_i}\left(\log \frac{f^N}{\bar{f}^N}\right)\cdot \E_{\theta^{(m)}}[b(v_i- v_{\theta^{(m)}(i)})f_m^{N,\theta}]-\nabla_{v_i} \left(\log \frac{f^N}{\bar{f}^N}\right)\cdot b(v_i; [\bar{f}]) f^N \Bigg)=:I+J.
\end{multline*}

To do further estimate, we introduce another copy $f_m^{N,\tilde{\theta}_i}$ of $f_m^{N,\theta}$ for each $i$, which is the same at $t_m$ but evolves with the coupled batch $\tilde{\theta}_i^{(m)}$.

For the term $I$, one has by Lemma \ref{lmm:coupledbatch} that
\begin{multline*}
I=-\frac{1}{2}\sum_{i=1}^N\int  \E \left(\nabla_{v_i}\left(\log \frac{f^N}{\bar{f}^N}\right) \otimes (\nabla_{v_i}f_m^{N,\theta}-\nabla_{v_i}f_m^{N,\tilde{\theta}_i}): A(v_i-v_{\theta^{(m)}(i)})\right) \\
-\frac{1}{2}\sum_{i=1}^N\int\Bigg(\nabla_{v_i}\left(\log \frac{f^N}{\bar{f}^N}\right)\otimes\nabla_{v_i}\log f^N :(\frac{1}{N-1}\sum_{j\neq i}A(v_i-v_j)) \\
- \nabla_{v_i}\left(\log \frac{f^N}{\bar{f}^N}\right)\otimes\nabla_{v_i}\log\bar{f}^N : A(v_i; \bar{f}) \Bigg) f^N=:I_{1}+I_{2}.
\end{multline*}
For term $J$, using the same argument, one has
\begin{multline*}
J=\sum_{i=1}^N \int \Bigg( \nabla_{v_i}\left(\log \frac{f^N}{\bar{f}^N}\right)\cdot \E \left[b(v_i- v_{\theta^{(m)}(i)})(f_m^{N,\theta}-f_m^{N,\tilde{\theta}_i})\right]\Bigg)\\
+\sum_{i=1}^N \int \nabla_{v_i}\left(\log \frac{f^N}{\bar{f}^N}\right)\cdot \Bigg(\frac{1}{N-1}\sum_{j\neq i}b(v_i-v_j)  - b(v_i; [\bar{f}]) \Bigg)f^N
=:J_1+J_2.
\end{multline*}

{\bf Step 2. Estimates of the terms}

The estimates of the $J$ terms are relatively easier.
The term $J_2$ can be estimated similarly using the usual tricks for propagation of chaos estimate in the drift (see, \cite{jabinquantitative}, for example). In fact,
\[
J_2\le  \sum_{i=1}^N \delta \int  \left|\nabla_{v_i} \log \frac{f^N}{\bar{f}^N}\right|^2 f^N +\sum_{i=1}^N C(\delta)\int \Big|\frac{1}{N-1}\sum_{j\neq i}b(v_i-v_j)  - b(v_i; [\bar{f}]) \Big|^2 f^N.
\]
The first term here will be absorbed by the entropy dissipation contained in $I$ term. The second term here will be controlled using Lemma \ref{lmm:youngIneq} and one has
\begin{multline*}
 \int \sum_{i=1}^N \Big|\frac{1}{N-1}\sum_{j\neq i}b(v_i-v_j)  - b(v_i; [\bar{f}]) \Big|^2 f^N
\le \frac{1}{\eta}\sum_{i=1}^N \frac{1}{N-1}H(f^N|\bar{f}^N) \\
+\frac{1}{\eta}\sum_{i=1}^N \frac{1}{N-1} \log \int \exp\bigg(\eta \frac{1}{N-1} \Big|\sum_{j\neq i}b(v_i-v_j)  - b(v_i; [\bar{f}]) \Big|^2\bigg)\bar{f}^N.
\end{multline*}
Since $b$ is bounded, one can apply Lemma \ref{lmm:concentration} directly for $\eta\lesssim \|b\|_{\infty}$ to obtain
\[
\int_{\R^{Nd}} \exp\bigg(\eta \frac{1}{N-1} \Big|\sum_{j\neq i}b(v_i-v_j)  - b(v_i; [\bar{f}]) \Big|^2\bigg)\bar{f}^N\le C
\]
independent of $N$. Hence,
\[
J_2\le \sum_{i=1}^N \delta \int_{\R^{Nd}}  \left|\nabla_{v_i} \log \frac{f^N}{\bar{f}^N}\right|^2 f^N+CH(f^N| \bar{f}^N)+C,
\]
where $C$ are independent of $N$.

For $J_1$ term, direct applying Young's inequality, one has
\begin{multline*}
J_1\le 
 \sum_{i=1}^N \delta \int  \left|\nabla_{v_i} \log \frac{f^N}{\bar{f}^N}\right|^2 \E(f_m^{N,\theta}+f_m^{N,\tilde{\theta}_i}) 
 +\sum_{i=1}^N C\|b\|_{\infty}^2 \int \E\Big( \frac{\left| f_m^{N,\theta}-f_m^{N,\tilde{\theta}_i}\right|^2}{f_m^{N,\theta}+f_m^{N,\tilde{\theta}_i}} \Big) \\
 \le 2\delta  \sum_{i=1}^N \int_{\R^{Nd}}  \left|\nabla_{v_i} \log \frac{f^N}{\bar{f}^N}\right|^2f^N+C\sum_{i=1}^N \int_{\R^{Nd}}  \E\left(\frac{\left| f_m^{N,\theta}-f_m^{N,\tilde{\theta}_i}\right|^2}{f_m^{N,\theta}+f_m^{N,\tilde{\theta}_i}} \right).
\end{multline*}
Applying Proposition \ref{pro:localerrors} for the local errors, one has
\begin{gather*}
\sum_{i=1}^N \int_{\R^{Nd}}  \E\left(\frac{\left| f_m^{N,\theta}-f_m^{N,\tilde{\theta}_i}\right|^2}{f_m^{N,\theta}+f_m^{N,\tilde{\theta}_i}} \right)\le C(T)N\Delta t^2,
\end{gather*}
where $C(T)$ is independent of $N$.

The term $I_2$ can be split into
\begin{multline*}
I_2=-\frac{1}{2}\sum_{i=1}^N\int\Bigg(\nabla_{v_i}\left(\log \frac{f^N}{\bar{f}^N}\right)\otimes \nabla_{v_i}\left(\log \frac{f^N}{\bar{f}^N}\right) : \frac{1}{N-1}\sum_{j\neq i}A(v_i-v_j)\Bigg)f^N \\
-\frac{1}{2}\sum_{i=1}^N\int\Bigg( \nabla_{v_i}\left(\log \frac{f^N}{\bar{f}^N}\right)\otimes\nabla_{v_i}\log\bar{f}^N : [\frac{1}{N-1}\sum_{j\neq i}A(v_i-v_j)-A(v_i; \bar{f})] \Bigg) f^N.
\end{multline*}
By the uniform positive definiteness of $A$, one then has
\begin{multline*}
I_2\le -\kappa \int_{\R^{Nd}} \sum_{i=1}^N \left|\nabla_{v_i} \log \frac{f^N}{\bar{f}^N}\right|^2 f^N
+\sum_{i=1}^N \delta \int  \left|\nabla_{v_i} \log \frac{f^N}{\bar{f}^N}\right|^2 f^N\\
+C(\delta)\sum_{i=1}^N \int_{\R^{Nd}} |\nabla_{v_i}\log\bar{f}^N|^2 \Big\|\frac{1}{N-1}\sum_{j\neq i}A(v_i-v_j)-A(v_i; \bar{f}) \Big\|^2 f^N.
\end{multline*}
Note that $|\nabla_{v_i}\log\bar{f}^N|
=|\nabla_{v_i}\log \bar{f}(t, v_i)|$ and we will show that $|\nabla_{v_i}\log \bar{f}(t, v_i)|\le C(1+|v_i|)$ in Section~\ref{sec:meanfieldproperty} under Assumptions \ref{ass:KA1} and \ref{ass:A2}. The power of $|v_i|$ is crucial here. With this bound, similar as we did for $J_2$, by combining with Lemma \ref{lmm:youngIneq} and Lemma \ref{lmm:concentration}, we then conclude that
\[
\sum_{i=1}^N \int_{\R^{Nd}} |\nabla_{v_i}\log\bar{f}^N|^2 \Big\|\frac{1}{N-1}\sum_{j\neq i}A(v_i-v_j)-A(v_i; \bar{f}) \Big\|^2 f^N 
\le CH(f^N |\bar{f}^N)+C,
\]
where $C$ is independent of $N$. See Proposition \ref{pro:I22} in Section~\ref{subsec:meandiffusion} for the details.

For $I_1$, similar as $J_1$, one again uses Young's inequality to obtain
\begin{multline*}
I_1\le 
 \sum_{i=1}^N 2 \delta \int_{\R^{Nd}}  \left|\nabla_{v_i} \log \frac{f^N}{\bar{f}^N}\right|^2 f^N \\
 +\sum_{i=1}^N C(\delta)\int_{\R^{Nd}} \E\left( \|A(v_i- v_{\theta^{(m)}(i)})\|^2  \frac{\left|\nabla_{v_i}f_m^{N,\theta}-\nabla_{v_i}f_m^{N,\tilde{\theta}_i}\right|^2}{f_m^{N,\theta}+f_m^{N,\tilde{\theta}_i}}  \right).
\end{multline*}
Applying Proposition \ref{pro:localerrors} for the local errors, one has
\begin{gather*}
\sum_{i=1}^N \int_{\R^{Nd}}  \E\left(\frac{\left| \nabla_{v_i}f_m^{N,\theta}-\nabla_{v_i}f_m^{N,\tilde{\theta}_i}\right|^2}{f_m^{N,\theta}+f_m^{N,\tilde{\theta}_i}} \right)\le C(T)N\Delta t,
\end{gather*}
where $C(T)$ is independent of $N$. This estimate is the term that does not have the optimal rate. See Section~\ref{subsec:localerr}
for the discussion.

Combining everything together, one has
\[
\frac{d}{dt}H(f^N | \bar{f}^N)
\le  (-\kappa+6\delta) \int_{\R^{Nd}} \sum_{i=1}^N \left|\nabla_{v_i} \log \frac{f^N}{\bar{f}^N}\right|^2 f^N+CH(f^N | \bar{f}^N)+C(1+N\Delta t).
\]
Taking $\delta=\kappa/6$ and applying G\"onwall's inequality, one finds that
\[
H(f^N | \bar{f}^N)\le C(T)(1+N\Delta t).
\]

Lastly, according to Lemma \ref{lmm:orderentropy}, one finds that
\[
H(f^{(j)} | \bar{f}^{\otimes j} )\le \frac{j}{N}H(f^N | \bar{f}^N)\le C(T)j\left(\frac{1}{N}+\Delta t\right).
\]
Since $\bar{f}$ has Gaussian tails, one can then apply the transport inequalities \cite{talagrand1996transportation,bobkov2001hypercontractivity} to obtain the bounds on the Wasserstein distances. Moreover, the Pinsker inequality \cite{pinsker1964information,bolley2005weighted} gives the bound on the total variation norm. The last claim is then justified.
\end{proof}

\section{The solution to the mean field equation}\label{sec:meanfieldproperty}

In this section, we present the estimates for the solution $\bar{f}$ to the mean field equation \eqref{eq:mckeanlandau}
\begin{gather}
\partial_t \bar{f}=-\nabla\cdot\left( K*\bar{f}(t, v) \bar{f}(t, v)\right)+\frac{1}{2}\nabla^2:\left(A*\bar{f}(t, v) \bar{f}(t, v) \right).
\end{gather}
Since we have assumed that $K$ and $A$ are bounded, and that $A$ is uniformly positive definite, it is standard to show that 
$\bar{f}$ exists and unique. Moreover, by bootstrapping, $\bar{f}$ is smooth. In this section, we mainly focus on the 
tail estimates of $\bar{f}$.

The following is a basic result for parabolic equations.
\begin{lemma}\label{lmm:taildensity}
Suppose that Assumption \ref{ass:KA1} for the kernels hold and that the initial density $f_0$ satisfies the Gaussian tail bounds in Assumption \ref{ass:initial}.
Then there exists $C_i=C_i(T)$ and $\alpha_i=\alpha_i(T)$ ($i=1,2$) such that
\begin{gather}
C_1\exp(-\alpha_1 |v|^2)\le  \bar{f}(t, v)\le C_2\exp(-\alpha_2 |v|^2).
\end{gather}
\end{lemma}

\begin{proof}[Sketch of the proof of Lemma \ref{lmm:taildensity}]
This result is standard so we only sketch the argument here. First, the existence and uniqueness of the solution $\bar{f}$ is straightforward with the strong assumptions on the kernels (see for example \cite{desvillettes2000spatially} for the argument that can be applied here).
Then, for the unique weak solution, $\bar{b}(t, v):=-K*\bar{f}(t,v)$ is bounded and Lipschitz while $\bar{a}(t, v):=\frac{1}{2}A*\bar{f}(t,v)$ is bounded and uniformly 
positive definite. Then, we regard $\bar{b}$ and $\bar{a}$ as given. 
Then, the solution $\bar{f}$ is also the solution to the linear Fokker-Planck equation
with coefficients $\bar{b}$ and $\bar{a}$:
\begin{gather}
\partial_t \bar{f}=\nabla\cdot(\bar{b}(t, v)\bar{f})+\nabla_v^2:(\bar{a}(t,v)\bar{f}).
\end{gather} 
The Green's function (transition probability) is then controlled by (see, for example, \cite{sheu1991some})
\begin{gather}
 C t^{-d/2}\exp\left(-\frac{\alpha |v-v'|^2}{t}\right) \le G(v, v'; t)\le C' t^{-d/2}\exp\left(-\frac{\alpha' |v-v'|^2}{t}\right)
\end{gather}
for some positive constants $C,C', \alpha, \alpha'$ (depending on $T$). Then, $\bar{f}(t, v)=\int f_0(v')G(v, v'; t)dv'$ has the desired control.
\end{proof}

The result in Lemma \ref{lmm:taildensity} clearly indicates that
\begin{gather}
|\log \bar{f}(t,v)|\le C(1+|v|^2),
\end{gather}
for some constant $C$ depending on $T$.

Next, we are devoted to obtain the Li--Yau type estimates of $\nabla \log\bar{f}$. See \cite{li1986parabolic} for the original Li--Yau estimates for heat equations. 
\begin{theorem}\label{pro:gradlogdensity}
We have the following bounds on the gradient of $\log \bar{f}$.
\begin{enumerate}[(i)]
\item  Suppose that Assumption \ref{ass:KA1} holds, and $\sup_v |\nabla_v\log f_0|/(1+\log \|f_0\|_{\infty}-\log f_0)<\infty$. Then it holds that
\begin{gather}
|\nabla_v \log \bar{f}| \le C(1+\log \|\bar{f}\|_{L^{\infty}([0, T]\times \R^d)}+|\log \bar{f}|)\le C'(1+|v|^2).
\end{gather}
\item Suppose that Assumption \ref{ass:KA1} and \ref{ass:A2} hold, and
$\sup_v |\nabla_v \log f_0|^2/(1+\log \|f_0\|_{\infty}-\log f_0)<\infty$. Then, it holds that
\begin{gather}
|\nabla_v \log \bar{f}| \le C\sqrt{1+\log \|\bar{f}\|_{L^{\infty}([0, T]\times \R^d)}+|\log \bar{f}|}\le C'(1+|v|).
\end{gather}
\end{enumerate}
\end{theorem}

\begin{proof}
Here, we regard $\bar{b}$ and $\bar{a}$ as known and consider the linear equation.
Let $M=\|\bar{f}\|_{L^{\infty}([0, T]\times \R^d)}$ and define
\[
u=\log \frac{\bar{f}}{M}\le 0.
\]
Then, $u$ satisfies
\[
\partial_t u=\bar{a}:(\nabla^2 u+\nabla u\otimes \nabla u)+(\bar{b}+\nabla\cdot\bar{a})\cdot \nabla u+(\nabla\cdot\bar{b}
+\nabla^2:\bar{a}).
\]

For notational convenience, we denote
\[
a=(a^{ij}):=\tilde{a},\quad b=(b^i):=\bar{b}+\nabla\cdot \bar{a},
\quad c:=\nabla\cdot\bar{b}+\nabla^2:\bar{a}.
\]
Moreover,  for $v=(v^1,\cdots, v^d)\in \R^d$,  we denote $u_i:=\frac{\partial u}{\partial v^i}$. Moreover, the Einstein summation convention is used, namely repeated indices that will be summed up. For example, $u_ku_k$ means $\sum_{k=1}^d u_k u_k$. 

(i)  Consider
\[
g=\frac{|\nabla u|^2}{(1-u)^2}=\frac{u_k u_k}{(1-u)^2}.
\]
Then,
\begin{multline*}
a^{ij}g_{ij}+b^i g_i-g_t=
2a^{ij}\frac{u_{ik}u_{jk}}{(1-u)^2}+8a^{ij}\frac{u_{ik}u_k u_j}{(1-u)^3}
+6a^{ij}\frac{u_ku_k u_i u_j}{(1-u)^4}
-4a^{ij}\frac{u_k u_{ik}u_j}{(1-u)^2}\\
-2a^{ij}\frac{u_ku_k u_i u_j}{(1-u)^3}-2c\frac{u_k u_k}{(1-u)^3}-2\frac{u_k(a_k^{ij}(u_{ij}+u_i u_j)+b^i_k u_i+c_k)}{(1-u)^2}
\end{multline*}
By observing that
\[
g_i=2\frac{u_k u_{ki}}{(1-u)^2}+2\frac{u_k u_k u_i}{(1-u)^3},
\]
one can then arrange the above expression into
\begin{gather*}
\begin{split}
\cL(g) :=&~a^{ij}g_{ij}+b^i g_i-g_t
-3a^{ij}\frac{u_j g_i}{1-u}+2a^{ij}u_j g_i\\
=&~2a^{ij}\frac{u_{ik}u_{jk}}{(1-u)^2}+2a^{ij}\frac{u_{ik}u_k u_j}{(1-u)^3}
+2a^{ij}\frac{u_ku_k u_i u_j}{(1-u)^3}\\
&~-2c\frac{u_k u_k}{(1-u)^3}-2\frac{u_k(a_k^{ij}(u_{ij}+u_i u_j)+b^i_k u_i+c_k)}{(1-u)^2}
\end{split}
\end{gather*}
Since
\[
a^{ij}\frac{u_{ik}u_{jk}}{(1-u)^2}+2a^{ij}\frac{u_{ik}u_k u_j}{(1-u)^3}
+a^{ij}\frac{u_ku_k u_i u_j}{(1-u)^4}\ge 0
\]
and $1-\frac{1}{1-u}\ge 0$, one then finds that
\begin{gather*}
\begin{split}
\cL g &\ge a^{ij}\frac{u_{ik}u_{jk}}{(1-u)^2}+a^{ij}\frac{u_ku_k u_i u_j}{(1-u)^3}
-\frac{2c}{(1-u)}g-2\frac{u_k(a_k^{ij}(u_{ij}+u_i u_j)+b^i_k u_i+c_k)}{(1-u)^2}\\
&\ge \frac{\kappa}{2}\frac{|\nabla^2 u|^2}{(1-u)^2}
+\frac{\kappa}{2}(1-u)g^2-2\|c\|_{\infty}g-C(M_1(g+1)+\|\nabla a\|_{\infty}(1-u)g),
\end{split}
\end{gather*}
where $M_1=\|\nabla a\|_{\infty}+\|\nabla b\|_{\infty}+\|\nabla c\|_{\infty}$.

Take a cutoff function $\tilde{\psi}(\cdot)$ defined on $[0,\infty)$  that is strictly positive on $[0,1)$, supported in $[0, 1]$ 
and is $1$ on $[0,1/2]$; and moreover, for any $\delta\in (0, 1)$, there exists $C_{\delta}>0$ such that
\[
|\tilde{\psi}'|\le C_{\delta}\tilde{\psi}^{1-\delta}.
\]
Such $\tilde{\psi}$ exists, for example, $\tilde{\psi}\sim \exp(-(1-r)^{-2})$
for $r\to 1^-$. 

Then, we take $\psi(v-v_0)=\tilde{\psi}(|v-v_0|)$ and compute
\[
\cL(\psi g)=\psi \cL g+g\cL \psi+2a^{ij}\psi_j g_i.
\]
Since $|\psi_i \psi_j/\psi|\le C$ by the condition on $\psi$, one then has
\[
2a^{ij}\psi_j g_i= 2a^{ij}\psi_j \psi^{-1} (\psi g)_i-2a^{ij}(\psi_i \psi_j/\psi)g
\ge 2a^{ij}\psi_j \psi^{-1} (\psi g)_i-C\|a\|_{\infty}g.
\]
Also,
\[
\begin{split}
g\cL\psi&=g[(a^{ij}\psi_{ij}+b^i\psi_i)-3a^{ij}\frac{u_j}{1-u}\psi_i
+2a^{ij}u_j\psi_i]\\
&\ge -C(g+|\nabla\psi|\sqrt{1-u}g^{3/2}).
\end{split}
\]
Using $|\nabla\psi|\le C_{\delta}\psi^{1-\delta}$, one finds that
$|\nabla\psi|\sqrt{1-u}g^{3/2}\le \epsilon\psi (1-u)g^2+C g$.
Then,
\[
\cL(g\psi)-2a^{ij}\psi_j \psi^{-1} (\psi g)_i\ge 
\frac{\kappa}{4}\psi (1-u)g^2-M(1-u)(g+1).
\]
Here, $M$ depends on the $W^{1,\infty}$ norms of $a,b,c$.

The function $\psi_j\psi^{-1}$ is well-defined in the interior of $B(v_0, 1)$.
Since $g\psi$ is zero on $[0, T]\times \partial B(v_0, 1)$, the maximum can be achieved in $(0, T]\times \mathrm{int}(B(v_0, 1))$ or achieved at $t=0$.
If the first case happens and the maximum point is $(t_1, v_1)$, then one has
\[
g(v_0, t)\le \psi(t_1, v_1)g(t_1, v_1)
\]
and
\[
0\ge \cL(g\psi)(t_1, v_1)-2a^{ij}\psi_j \psi^{-1} (\psi g)_i(t_1, v_1) \ge \frac{\kappa}{4}\psi (1-u)g^2-M(1-u)(g+1) \Big|_{(t_1, v_1)}.
\]
It follows then
\[
\psi(t_1, v_1)g(t_1, v_1)\le \frac{4M}{\kappa}\cdot\frac{g(t_1, v_1)+1}{g(t_1, v_1)}.
\]
If $g(t_1, v_1)\le 1$, then $\psi(t_1, v_1)g(t_1, v_1)\le 1$. Otherwise, the right-hand side is bounded by 
$8M/\kappa$. Hence, for all $t\in [0, T]$, one has in this case
\[
g(v_0, t)\le \psi(t_1, v_1)g(t_1, v_1)\le \max(1, \frac{8M}{\kappa}).
\]
If the maximum is achieved at the initial point, then
\[
g(v_0, t)\le \sup_{B(v_1,1)}g(\cdot, 0).
\]
The result then holds.

(ii)  The computation is similar as part (i). Here, we consider the auxiliary 
function 
\[
g=\frac{|\nabla u|^2}{1-u}.
\]
By the conditions given
\[
\|g(\cdot, t=0)\|_{\infty}<+\infty.
\]
Direct calculation reveals that
\begin{gather*}
\begin{split}
a^{ij}g_{ij}+b^i g_i-g_t
= &~\frac{2a^{ij}u_{ki}u_{kj}}{1-u}+\frac{4a^{ij}u_k u_{ki}u_j}{(1-u)^2}
-\frac{4a^{ij}u_k u_{ki}u_j}{1-u}\\
&+\frac{2a^{ij}u_ku_k u_iu_j}{(1-u)^3}
-\frac{a^{ij}u_ku_k u_iu_j}{(1-u)^2}-R\\
=&~\frac{2a^{ij}u_{ki}u_{kj}}{1-u}+\frac{a^{ij}u_ku_k u_i u_j}{(1-u)^2}+\frac{2ua^{ij}u_j}{1-u}g_i-R,
\end{split}
\end{gather*}
where
\begin{gather*}
R=c\frac{u_ku_k}{(1-u)^2}+\frac{2u_k}{1-u}(a_k^{ii}u_{ij}+a_k^{ij}u_iu_j
+b_k^iu_i+c_k).
\end{gather*}
We then consider
\[
\cL(g):=a^{ij}g_{ij}+b^i g_i-g_t-\frac{2ua^{ij}u_j}{1-u}g_i
=\frac{2a^{ij}u_{ki}u_{kj}}{1-u}+\frac{a^{ij}u_ku_k u_i u_j}{(1-u)^2}-R,
\]
and find that
\[
\cL(g)\ge (\kappa-\delta)\frac{|\nabla^2 u|^2}{1-u}
+\kappa g^2-Mg-\|\nabla c\|_{\infty}\frac{\sqrt{g}}{\sqrt{1-u}}
-\frac{2u_k}{1-u}a_k^{ij}u_iu_j,
\]
where $M=\|\nabla a\|_{\infty}+\|\nabla b\|_{\infty}$.

Since $a=( A*\bar{f})(v)$, by Assumption \ref{ass:A2}, one immediately has
\[
\sup_{t\le T}\|\nabla a(v, t)\|\le C(T)(1+|v|)^{-1}.
\]
Consequently, 
\[
\|a_k^{ij} \sqrt{1-u}\|_{\infty}<\infty.
\]
Hence, 
\[
\cL(g)\ge (\kappa-\delta)\frac{|\nabla^2 u|^2}{1-u}
+\kappa g^2-Mg-\|\nabla c\|_{\infty}\frac{\sqrt{g}}{\sqrt{1-u}}
-M g^{3/2}.
\]

Now, for fixed $v_0$, we take the cutoff function to be 
\[
\psi(v)=\tilde{\psi}((1+2|v_0|)^{-1}|x|),
\]
where $\tilde{\psi}$ is the same as in Part (i). 
Then, $\psi$ is supported in $B(0,1+2|v_0|)$, and equals $1$ inside $B(0, 1/2+|v_0|)$ including $v_0$.

Then, one has
\[
\cL(\psi g)=\psi \cL(g)+g\cL(\psi)+2a^{ij}\psi_j g_i.
\]
Similarly as above, $|\psi_i \psi_j|/\psi\le C(1+2|v_0|)^{-2}\le C$ so that
\[
2a^{ij}\psi_j g_i= 2a^{ij}\psi_j \psi^{-1} (\psi g)_i-2a^{ij}(\psi_i \psi_j/\psi)g
\ge 2a^{ij}\psi_j \psi^{-1} (\psi g)_i-C\|a\|_{\infty}g.
\]
For $\cL(\psi)$, we only need to consider the term:
\[
|-\frac{2ua^{ij}u_j}{1-u}\psi_i|
\le \frac{C|u|}{1-u}|\nabla u|(1+2|v_0|)^{-1} 1_{|v|\le 1+2|v_0|}.
\]
Note that for $|v|\le 1+2|v_0|$, $\sqrt{1-u(v, t)}\le C(T)(1+|v|)$ and thus
\[
(1+2|v_0|)^{-1} 1_{|v|\le 1+2|v_0|}\le C\frac{1}{\sqrt{1-u(v, t)}}.
\]
This is the reason why we take such a cutoff.  Consequently, this term is controlled by $Cg^{1/2}$. 
Hence,
\[
\cL(g\psi)-2a^{ij}\psi_j \psi^{-1} (\psi g)_i \ge \psi\frac{\kappa}{2}\left(\frac{|\nabla^2 u|^2}{1-u}
+ g^2\right)-C[g^{3/2}+1].
\]
Performing the same argument as above gives that 
$g(v_0, t)$ is bounded by a constant (independent of $v_0$).
\end{proof}

\section{Estimates of the density for the particle system}\label{sec:Npardensity}

In this section, we aim to estimate the integral version of $|\partial^{\alpha} \log \tilde{f}^{\b{\theta}}|^p$
for the density of the particle system with various derivatives $\alpha$. 
The main technical difficulty is that we need the estimates to scale linearly with the particle number. 
As we shall see, we will introduce a new advection term to address the estimates for the first derivatives.

\subsection{Main estimates for derivatives of logarithmic density}\label{sec:resultsdensity}

Usually, the pointwise estimates of such quantities are based on the Li--Yau estimate. For multiplicative noises, we find it would be more convenient to estimate the integrated version directly.

The first class of quantities we are interested in is the following one, which is the generalized Fisher information:
\begin{gather}
I_p^{i}(t):=\int_{\R^{Nd}}  |\nabla_{v_i} \log \tilde{f}^{\b{\theta}}(t, v)|^p\tilde{f}^{\b{\theta}}(t, v)dv.
\end{gather}
The second class of quantities is following second order derivative estimate
\begin{gather}
H_p^{i}(t):=\int_{\R^{Nd}}  |\nabla_{v_i}^2 \log \tilde{f}^{\b{\theta}}(t, v)|^p
\tilde{f}^{\b{\theta}}(t, v)dv.
\end{gather}
We will consider the following quantity
\begin{gather}
\tilde{K}^i(t)=H_2^{i}(t)+\gamma (I_{2}^{i}(t)+I_{4}^{i}(t)).
\end{gather}

The following estimates are some important local estimates for the particle density. The proof of this lemma is quite involved, and we leave it to Subsection~\ref{subsec:proofparticledensity}.
\begin{lemma}\label{lmm:onestepstability}
For any $i\in \{1,\cdots, N\}$, let $\theta(i)$ being its batchmate at $t_m$.
Then it holds for some $\gamma>0$ and all $t\in [t_m, t_{m+1}]$ that
\begin{gather}
I_p^{i}(t)+I_p^{\theta(i)}(t)\le e^{C(t-t_m)}(I_p^{i}(t_m)+I_p^{\theta(i)}(t_m))+C(t-t_m),\\
\tilde{K}^i(t)+\tilde{K}^{\theta(i)}(t)\le e^{C(t-t_m)}(\tilde{K}^{i}(t_m)+\tilde{K}^{\theta(i)}(t_m))+C(t-t_m),
\end{gather}
where $C$ is independent of $N$. 
\end{lemma}

As a direct consequence of the above results, we have the following bounds.
\begin{theorem}\label{thm:maincontrol}
For any $p\ge 2$, there exists a constant $C=C(p, T, d)$ independent of $N$ such that for all $t\le T$, one has
\begin{gather}\label{eq:pfinsherboundN}
\int_{\R^{Nd}} \sum_{i=1}^N |\nabla_{v_i} \log \tilde{f}^{\b{\theta}}(t, v)|^p \tilde{f}^{\b{\theta}} dv\le C N,
\end{gather}
and that
\begin{gather}\label{eq:HessianboundN}
\int_{\R^{Nd}} \sum_{i=1}^N |\nabla_{v_i}^2 \log \tilde{f}^{\b{\theta}}(t, v)|^2 \tilde{f}^{\b{\theta}}(t, v)dv\le C N.
\end{gather}
\end{theorem}

\begin{proof}
This is in fact a direct consequence of Lemma \ref{lmm:onestepstability}.
Taking the sum over $i$, one finds that
\begin{multline*}
\sum_{i=1}^N\int_{\R^{Nd}} |\nabla_{v_i} \log \tilde{f}^{\b{\theta}}(t, v)|^p\tilde{f}^{\b{\theta}}(t, v)dv \\
\le e^{C(t-t_m)}\sum_{i=1}^N\int_{\R^{Nd}} |\nabla_{v_i} \log \tilde{f}^{\b{\theta}}(t_m, v)|^p\tilde{f}^{\b{\theta}}(t_m, v)dv
+CN(t-t_m),
\end{multline*}
where $C$ is independent of $N$. Iterating this gives
\begin{multline*}
\sum_{i=1}^N\int_{\R^{Nd}} |\nabla_{v_i} \log \tilde{f}^{\b{\theta}}(t, v)|^p\tilde{f}^{\b{\theta}}(t, v)dv \\
\le e^{Ct}\sum_{i=1}^N\int_{\R^{Nd}} |\nabla_{v_i} \log \tilde{f}^{\b{\theta}}(0, v)|^p\tilde{f}^{\b{\theta}}(0, v)dv
+CN\left(t-t_m+\sum_{j=1}^m e^{C(t-t_j)}\eta \right).
\end{multline*}
Since $\int |\nabla_{v_i} \log \tilde{f}^{\b{\theta}}(0, v)|^p\tilde{f}^{\b{\theta}}(0, v)dv
=\int_{\R^d} |\nabla_v f_0(v)|^p f_0(v)dv$ is independent of $N$, the claim then follows.

The claim for second order derivatives can be similarly obtained.
\end{proof}

As a direct consequence of the above results, one has
\begin{corollary}\label{cor:expectedparticledensity}
There exists a constant $C=C(T)$ independent of $N$ and $\theta$ such that
\begin{gather}
\sum_{i=1}^N \int_{\R^{Nd}} |\nabla_{v_i}\log f_m^{N,\theta}|^p f_m^{N,\theta}\le C(T)N
\end{gather}
and 
\begin{gather}\label{eq:integralsecondorder}
\sum_{i=1}^N \int_{\R^{Nd}} |\nabla_{v_i}^2\log f_m^{N,\theta}|^2 f_m^{N,\theta}\le C(T)N.
\end{gather}
Moreover, the the same estimates hold for $f^N$ as well.
\end{corollary}
\begin{proof}
The proofs for $f_m^{N,\theta}$ and $f^N$ are the same. We take $f_m^{N,\theta}$ as the example. First note that
\[
\int_{\R^{Nd}} |\nabla_{v_i}\log f_m^{N,\theta}|^p f_m^{N,\theta}
=\int_{\R^{Nd}} \frac{|\nabla_{v_i} f_m^{N,\theta}|^p }{|f_m^{N,\theta}|^{p-1}}.
\]
The mapping $(f, m)\mapsto \frac{|m|^p}{f^{p-1}}$ is convex for $p\ge 1$.
Then, as a consequence of $f_m^{N,\theta}=\mathbb{E}[\tilde{f}^{\b{\theta}}|\theta_m]$ and the Jensen's inequality, the first inequality holds by \eqref{eq:pfinsherboundN}.

For the second inequality, similar argument yields that
\[
\begin{split}
\sum_{i=1}^N \int_{\R^{Nd}} |\nabla_{v_i}^2\log f_m^{N,\theta}|^2 f_m^{N,\theta}
&\le \sum_{i=1}^N \int_{\R^{Nd}} \frac{|\nabla_{v_i}^2 f_m^{N,\theta}|^p }{|f_m^{N,\theta}|^{p-1}}
+\sum_{i=1}^N \int_{\R^{Nd}} \frac{|\nabla_{v_i} f_m^{N,\theta}|^{2p} }{|f_m^{N,\theta}|^{2p-1}}\\
& \le \sum_{i=1}^N \E\left( \int_{\R^{Nd}} \frac{|\nabla_{v_i}^2 f^{\b{\theta}}|^p }{|f^{\b{\theta}}|^{p-1}} | \theta_m \right)+CN.
\end{split}
\]
By \eqref{eq:HessianboundN} and \eqref{eq:pfinsherboundN},
the first term is also controlled by $CN$.
\end{proof}

Another consequence of Theorem \ref{thm:maincontrol} and the an intermediate step in the proof, one also has the following estimates.
\begin{corollary}\label{cor:diss_est}
There exists a constant $C=C(T)$ independent of $N$ and $\theta$ such that
\begin{gather}
\int_0^T\int_{\R^{Nd}}\sum_{i=1}^N
|\nabla_{v_i}\log \tilde{f}^{\b{\theta}}(t, v)|^{p-2}\left(\sum_{j=1}^N|\nabla_{v_i}\nabla_{v_j} \log \tilde{f}^{\b{\theta}}(t, v)|^2\right) 
\tilde{f}^{\b{\theta}}(t, v)\le CN,
\end{gather}
and that
\begin{gather}
\int_0^T\int_{\R^{Nd}}\sum_{i=1}^N\left(\sum_{j=1}^N|\nabla_{v_i}\nabla_{v_j}^2 \log \tilde{f}^{\b{\theta}}(t, v)|^2\right) 
\tilde{f}^{\b{\theta}}(t, v)\le CN.
\end{gather}

\end{corollary}
The proof of this corollary is deferred to the next Subsection~\ref{subsec:proofparticledensity}.

\subsection{The proofs for the estimates of the particle density}\label{subsec:proofparticledensity}

In this section, we give the proofs to the results in Subsection~\ref{sec:resultsdensity}.  We first recall \eqref{eq:jointrbm} and introduce some notation to simplify the presentation. 
We consider $t\in (t_m, t_{m+1}]$. During this interval, the $N$ particles are divided into $N/2$ groups, and
$i$ and $\theta^{(m)}(i)$ interact with each other without affecting other particles. Hence, we can order the particles
to be $\{\sigma(1), \sigma(2), \cdots, \sigma(N)\}$ such that $\sigma(2k-1)$ and $\sigma(2k)$ are in the same group.
Introduce 
\begin{gather}
x_k:=[v_{\sigma(2k-1)}, v_{\sigma(2k)}]^T,
\end{gather}
and
\begin{equation}
\begin{split}
& \tilde{a}(x_k):=
\frac{1}{2}\begin{bmatrix}
A(v_{\sigma(2k-1)}, v_{\sigma(2k)}) & \\
& A(v_{\sigma(2k)}, v_{\sigma(2k-1)})
\end{bmatrix}
\in \R^{2d\times 2d},  \\
&\tilde{b}(x_k):=
-\begin{bmatrix}
K(v_{\sigma(2k-1)}, v_{\sigma(2k)})\\
K(v_{\sigma(2k)}, v_{\sigma(2k-1)})
\end{bmatrix}
\in \R^{2d}.
\end{split}
\end{equation}
Then, \eqref{eq:jointrbm} is rewritten as
\begin{gather}
\partial_t \tilde{f}^{\b{\theta}}=\sum_{k=1}^{N/2}\nabla_{x_k}\cdot(\tilde{b}(x_k)\tilde{f}^{\b{\theta}})
+\sum_{k=1}^{N/2}\nabla_{x_k}^2:(\tilde{a}(x_k)\tilde{f}^{\b{\theta}}).
\end{gather}
Denote
\begin{gather}
h=\log \tilde{f}^{\b{\theta}}.
\end{gather}
Then, it holds that
\begin{multline}
\partial_t h=\sum_{k=1}^{N/2}\Big[\tilde{a}(x_k):\nabla_{x_k}^2h+\tilde{a}(x_k):\nabla_{x_k} h\otimes \nabla_{x_k} h\\
+(\tilde{b}(x_k)+2\nabla_{x_k}\cdot \tilde{a}(x_k))\cdot \nabla_{x_k}h+\nabla_{x_k}\cdot \tilde{b}(x_k)
+\nabla_{x_k}^2: \tilde{a}(x_k)\Big].
\end{multline}

We set $x^k:=(x_1^k, x_2^k)$ so that $x_1^k=v_{\sigma(2k-1)}$ and $x_2^k=v_{\sigma(2k)}$.
Denote
$ \tilde{a}^k:=\tilde{a}(x_k)$ and $\tilde{b}^k:=\tilde{b}(x_k)$.
Moreover, for a function $g$ and a vector field $w$, denote
\[
 g_k=\begin{pmatrix}\nabla_{v_{\sigma(2k-1)}}g\\
  \nabla_{v_{\sigma(2k)}}g
  \end{pmatrix}
  :=\nabla_{x_k}g, \quad w_{,k}:=\nabla_{x_k}\cdot w.
\]
Then, $\tilde{a}_{,kk}^k=\nabla_{x_k}^2: \tilde{a}(x_k)$.
We also introduce
\begin{gather}
a^k=\tilde{a}^k,\quad b^k=\tilde{b}^k+\tilde{a}^k_{,k},
\quad c^k=\tilde{b}^k_{,k}+\tilde{a}^k_{,kk}.
\end{gather}
Then, 
\begin{gather}
h_t=\sum_{i=1}^{N/2} (a^i:h_{ii}+a^i: h_i\otimes h_i+b^i\cdot h_i+c^i)
\end{gather}
In the proof below, we will also consider the derivatives on $v_{\alpha}$ for $1\le \alpha\le N$.
For the convenience, we will use Greek subindices $\alpha, \beta$ to represent such derivatives. For example, $h_{\alpha}$ means $\nabla_{v_{\alpha}}h$ and $h_{\alpha,\alpha}$ means $\nabla_{v_{\alpha}}\cdot h_{\alpha}=\Delta_{\alpha}h$.

\begin{proof}[Proof of Lemma \ref{lmm:onestepstability}]

(i) Consider the first order derivatives. 
Let $\mathcal{I}_{k}=\{\alpha, \tilde{\alpha}\}$ be the set of two indices $\sigma(2k-1)$ and $\sigma(2k)$ associated with $x_k$.
Define
\begin{gather}
u^k:=e^h \sum_{\alpha\in \cI_k}| h_{\alpha}|^p.
\end{gather}

It is straightforward to compute that (see Appendix \ref{app:missingdetails} for the details)
\begin{multline}\label{eq:eqnforu1}
u^k_t-\sum_{i=1}^{N/2} (\tilde{b}^i u^k)_i-\sum_i(\tilde{a}^i u^k)_{,ii}
=pe^h \sum_{\alpha\in \cI_k}|h_{\alpha}|^{p-2}h_{\alpha}\cdot \Big(a_{\alpha}^k:h_{kk}
+a_{\alpha}^k: h_k\otimes h_k+b_{\alpha}^k\cdot h_k
+c_{\alpha}\Big)\\
-pe^h\sum_{\alpha\in \cI_k} |h_{\alpha}|^{p-2}\sum_{i=1}^{N/2} a^i:_i [(I+(p-2)\frac{h_{\alpha}\otimes h_{\alpha}}{|h_{\alpha}|^2}):_{\alpha} h_{\alpha i}\otimes h_{\alpha i}].
\end{multline}

The last term is the good term which gives the dissipation.  In fact, using the fact that $a^i\succeq \kappa I_{2d}$, one has
\begin{gather*}
\begin{split}
a^i:_i [(I+(p-2)\frac{h_{\alpha}\otimes h_{\alpha}}{|h_{\alpha}|^2}):_{\alpha} h_{\alpha i}\otimes h_{\alpha i}]
&\ge \sum_{\lambda} \lambda a^i:[e_{\lambda}\cdot\nabla_{\alpha}h_i\otimes (e_{\lambda}\cdot\nabla_{\alpha}h_i)]\\
&\ge \sum_{\beta\in \cI_i}\sum_{\lambda} (p-1) \kappa  |e_{\lambda}\cdot h_{\alpha\beta}|^2\\
&=(p-1)\kappa \sum_{\beta\in \cI_i} |h_{\alpha \beta}|^2,
\end{split}
\end{gather*}
where $(\lambda, e_{\lambda})$ are the eigenpairs of the matrix 
$I+(p-2)\frac{h_{\alpha}\otimes h_{\alpha}}{|h_{\alpha}|^2}$ and $|M|=\sqrt{\sum_{i,j} M_{ij}^2}$ is the Frobenius norm of $M$.

However,  the nonlinear term $a^k_{\alpha}: h_k\otimes h_k$ is the trouble term because the power is more than we can control. 
Our observation here is that
\begin{gather}\label{eq:auxadvection}
(h_{\alpha} e^h |h_{\alpha}|^p)_{,\alpha}=e^h 
|h_{\alpha}|^{p+2}+e^h|h_{\alpha}|^p h_{\alpha, \alpha}+p e^h
|h_{\alpha}|^{p-2} h_{\alpha}\otimes h_{\alpha} : h_{\alpha \alpha}.
\end{gather}
Then, we add this transport term with a coefficient $\gamma$ small enough into the inequality. Then, using Young's inequality, one has
\begin{gather}\label{eq:keyeqn1}
\begin{split}
& u^k_t-\gamma \sum_{\alpha\in \cI_k}(h_{\alpha} e^h |h_{\alpha}|^p)_{,\alpha}-\sum_i (\tilde{b}^i u^k)_i-\sum_i(\tilde{a}^i u^k)_{,ii}\\
& \le -\tilde{\kappa} e^h\sum_{\alpha\in \cI_k} |h_{\alpha}|^{p-2}\sum_{\beta=1}^{N}  |h_{\alpha\beta}|^2
-\tilde{\kappa}e^h \sum_{\alpha\in \cI_k} |h_{\alpha}|^{p+2}
+Ce^h\sum_{\alpha\in \cI_k} (|h_{\alpha}|^{p}+1),
\end{split}
\end{gather}
where $C$ is independent of $N$.
The details for \eqref{eq:eqnforu1} and \eqref{eq:keyeqn1} can  be found in Appendix~\ref{app:missingdetails}.

Taking the integral, one then finds that
\begin{gather}\label{eq:inter1}
\frac{d}{dt}\int e^h \sum_{\alpha\in \cI_k} |h_{\alpha}|^p+\tilde{\kappa} \int e^h\sum_{\alpha\in \cI_k} |h_{\alpha}|^{p-2}\sum_{\beta=1}^{N}  |h_{\alpha\beta}|^2
\le C\left(\int e^h\sum_{\alpha\in \cI_k} |h_{\alpha}|^p+1 \right).
\end{gather}
 Integrating over this inequality gives the first result.

(ii) We move to the second order derivatives. Define
\begin{gather}
u^k:=e^h \sum_{\alpha\in \cI_k} |h_{\alpha\alpha}|^{2}.
\end{gather}
By similar computation to \eqref{eq:keyeqn1}, one has
\begin{multline}\label{eq:hessianderivative}
u_t^k-\sum_{i=1}^{N/2}((\tilde{a}^i u^k )_{,ii}+(\tilde{b}^i u^k )_{,i})
=u_t^k-\sum_{i=1}^{N/2}(a^i:u^k_{ii}+b^i\cdot u^k_i+c^iu^k )\\
=2 e^h\sum_{\alpha\in \cI_k} 
\Big[ h_{\alpha\alpha}\cdot (\sum_i 2a^i: h_{i\alpha }\otimes h_{i\alpha})
-\sum_ia^i: h_{\alpha\alpha i}\otimes_i h_{\alpha\alpha i}+R_{k,\alpha} \Big],
\end{multline}
where
\begin{multline*}
R_{k,\alpha}
=h_{\alpha\alpha}
\cdot(a^k_{\alpha\alpha}:h_{kk}
+2a^k_{\alpha}:h_{kk \alpha}
+a^k_{\alpha\alpha}: h_k\otimes h_k\\
+4a^k_{\alpha}: h_k\otimes h_{k \alpha}
+b^k_{\alpha\alpha}\cdot h_k+2b^k_{\alpha}\cdot h_{k \alpha}
+c^k_{\alpha\alpha}).
\end{multline*}
Noting that $a^k$ is block-diagonal so that one has estimates like  and by Young's inequality, one finds that
\[
\sum_{\alpha\in \cI_k} R_{k,\alpha}
\le \sum_{\alpha, \beta \in \cI_k}(\delta |h_{\alpha \alpha \beta}|^2
+C( |h_{\alpha \beta}|^2
+ |h_{\alpha}|^2|h_{\alpha \beta}|^2+|h_\beta|^4+1)).
\]
Consider the first term on the right-hand side in \eqref{eq:hessianderivative}
\begin{multline}\label{eq:nonlinearaux1}
2e^h h_{\alpha\alpha }: (\sum_{i=1}^{N/2} 2a^i: h_{i \alpha}\otimes h_{i \alpha})
=4\{e^h h_{k_{\ell}}\cdot (\sum_{i=1}^{N/2} a^i: h_{i \alpha}\otimes h_{i \alpha})\}_{\alpha}\\
-4e^h h_{\alpha}\otimes h_{\alpha}:(\sum_{i=1}^{N/2} a^i: h_{i \alpha}\otimes h_{i \alpha})\\
-8 e^hh_{\alpha}
\cdot \sum_{i=1}^{N/2}a^i: h_{i \alpha}\otimes h_{i \alpha\alpha }
-4e^hh_{\alpha}\otimes a^k_{\alpha}:(h_{k \alpha}\otimes h_{k \alpha}).
\end{multline}
The first term will integrate to zero while the second term has a definite sign. 
For the third term, noting that $a^i$ is block-diagonal, one has
\[
-8 e^hh_{\alpha}
\cdot \sum_{i=1}^{N/2}a^i: h_{i \alpha}\otimes h_{i \alpha\alpha}
\le e^h\delta \sum_{i=1}^{N/2} a^i: h_{i \alpha\alpha} \otimes_i h_{i \alpha\alpha}
+Ce^h \sum_{j=1}^{N}|h_{\alpha}|^2 |\nabla_{v_j}h_{\alpha}|^2.
\]
The last term on the right-hand side of \eqref{eq:nonlinearaux1} is controlled by $\sum_{\alpha,\beta\in \cI_k}C( |h_{\alpha \beta}|^2
+ |h_{\alpha}|^2|h_{\alpha\beta}|^2)$.

We take the integral of \eqref{eq:hessianderivative}. Note that $a^i$ is block-diagonal so that 
 $\sum_{i=1}^{N/2}a^i: h_{\alpha\alpha i}\otimes h_{\alpha\alpha i} \ge \kappa \sum_{\alpha\in \cI_k} \sum_{j=1}^N |\nabla_{v_j}h_{\alpha\alpha}|^2$. Taking $\delta$ small enough, one finds that
\begin{multline}\label{eq:inter2}
\frac{d}{dt}\int e^h \sum_{\alpha\in \cI_k} |h_{\alpha\alpha}|^{2}
\le -(\kappa-2\delta)  \sum_{\alpha\in \cI_k} \sum_{j=1}^N\int e^h |\nabla_{v_j}h_{\alpha\alpha}|^2+ \\
C\int e^h\sum_{\alpha\in \cI_k} \sum_{j=1}^{N} |h_{\alpha}|^2 |\nabla_{v_j}h_{\alpha}|^2
+C(\int e^h \sum_{\alpha,\beta\in \cI_k} |h_{\alpha \beta}|^2+1)+\int e^h\sum_{\alpha\in \cI_k}|h_{\alpha}|^4
\end{multline}

Combining \eqref{eq:inter1} and \eqref{eq:inter2}, one finds for some $\gamma>0$ that
\begin{multline}\label{eq:inter3}
\frac{d}{dt}\left(\int e^h \sum_{\alpha\in \cI_k} |h_{\alpha\alpha}|^{2}+\gamma \int e^h(|h_k|^2+|h_k|^4)\right)\\
+\tilde{\kappa} \sum_{\alpha\in \cI_k} \sum_{j=1}^N\int e^h |\nabla_{v_j}h_{\alpha\alpha}|^2 \le C\left(\gamma \int e^h \sum_{\alpha\in \cI_k} |(h_{\alpha}|^2+|h_\alpha |^4)+1\right).
\end{multline}

Applying Gr\"onwall's inequality to \eqref{eq:inter1} and \eqref{eq:inter3} from $t_m$ to $t\in (t_m, t_{m+1}]$ gives the desired results.
\end{proof}

Next, we give the proof to Corollary \ref{cor:diss_est}.
\begin{proof}[Proof of Corollary \ref{cor:diss_est}]

For the first assertion, we use \eqref{eq:inter1}, which is equivalent to
\begin{multline*}
\frac{d}{dt}\int \tilde{f}^{\b{\theta}}\left(|\nabla_{v_i}\log \tilde{f}^{\b{\theta}}|^p
+|\nabla_{v_{\theta(i)}}\log \tilde{f}^{\b{\theta}}|^p\right)
+\tilde{\kappa}\int \tilde{f}^{\b{\theta}} \Big(|\nabla_{v_i}\log \tilde{f}^{\b{\theta}}|^{p-2}
\sum_{j=1}^N|\nabla_{v_j}\nabla_{v_i} \log \tilde{f}^{\b{\theta}}|^2\\
+|\nabla_{v_{\theta(i)}}\log \tilde{f}^{\b{\theta}}|^{p-2}\sum_{j=1}^N|\nabla_{v_j}\nabla_{v_{\theta(i)}} \log \tilde{f}^{\b{\theta}}|^2\Big)
\le C\left(\int \tilde{f}^{\b{\theta}}(|\nabla_{v_i}\log \tilde{f}^{\b{\theta}}|^p
+|\nabla_{v_{\theta(i)}}\log \tilde{f}^{\b{\theta}}|^p)+1\right).
\end{multline*}
Taking the sum of over $ 1\le i\le N/2$, one has
\begin{multline*}
\frac{d}{dt}\int_{\R^{Nd}} \tilde{f}^{\b{\theta}}\sum_{j=1}^N |\nabla_{v_j}\log \tilde{f}^{\b{\theta}}|^p
+\tilde{\kappa}\int_{\R^{Nd}} \tilde{f}^{\b{\theta}} \sum_{j=1}^N |\nabla_{v_j}\log \tilde{f}^{\b{\theta}}|^{p-2}
\sum_{j=1}^N|\nabla_{v_j}\nabla_{v_i} \log \tilde{f}^{\b{\theta}}|^2\\
\le C\left(\int_{\R^{Nd}} \tilde{f}^{\b{\theta}}\sum_{j=1}^N |\nabla_{v_j}\log \tilde{f}^{\b{\theta}}|^p+N\right).
\end{multline*}
This inequality holds for all $t\in (t_m, t_{m+1})$ and all $m$ so that one can take the integral from $0$
to $T$ to obtain the first desired result, using the control in Theorem \ref{thm:maincontrol}.

The second desired result can be obtained by \eqref{eq:inter3} and similar argument.
\end{proof}

\subsection{Estimates for the local error terms}\label{subsec:localerr}

In our problem, the diffusion coefficients for $f_m^{N,\theta}$ and $f_m^{N, \tilde{\theta}_i}$
are different so the usual $f$-divergence and the relative Fisher information would come across difficulty.
Instead, we adopt the symmetrized version for the local error.

\begin{proposition}\label{pro:localerrors}
For any $t_m \le t\le \min(t_{m+1}, T)$,  the following local estimates hold
\begin{subequations}
\begin{align}
&\E \sum_{i=1}^N \int_{\R^{Nd}}\frac{|f_m^{N,\theta}-f_m^{N, \tilde{\theta}_i}|^2}{f_m^{N,\theta}+f_m^{N, \tilde{\theta}_i}}
\le C N (t-t_m)^2,\\
&\E \sum_{i=1}^N \int_{\R^{Nd}}\frac{|\nabla_{v_i}f_m^{N,\theta}-\nabla_{v_i}f_m^{N, \tilde{\theta}_i}|^2}{f_m^{N,\theta}+f_m^{N, \tilde{\theta}_i}}
\le C N (t-t_m),
\end{align}
\end{subequations}
where $C$ is independent of $N$.
\end{proposition}

\begin{proof}
Denote
\[
E_1(t):=\E \sum_{i=1}^N \int_{\R^{Nd}}\frac{|f_m^{N,\theta}-f_m^{N, \tilde{\theta}_i}|^2}{f_m^{N,\theta}+f_m^{N, \tilde{\theta}_i}},
\quad E_2(t):=\E \sum_{i=1}^N \int_{\R^{Nd}}\frac{|\nabla_{v_i}f_m^{N,\theta}-\nabla_{v_i}f_m^{N, \tilde{\theta}_i}|^2}{f_m^{N,\theta}+f_m^{N, \tilde{\theta}_i}}.
\]
Clearly, $E_1(t_m)=0$ and $E_2(t_m)=0$. Below, we will basically show that
\[
\frac{d}{dt}E_{\ell}(t)\lesssim \sqrt{E_{\ell}(t)}\sqrt{N},
\quad \ell=1,2.
\]
This will then give the result.

For the purpose, we first fix $i$ and denote
\[
f:= f_m^{N, \theta}, \quad \tilde{f}:=f_m^{N, \tilde{\theta}_i}.
\]
We use the same conventions as in Section~\ref{subsec:proofparticledensity} to simplify the notations for derivatives.
Also denote $\delta f:= f-\tilde{f}$ and $\delta a^i=a^i-\tilde{a}^i$. 
Then, it is straightforward to compute
\begin{multline*}
\frac{d}{dt}\int_{\R^{Nd}} \frac{|\delta f|^2}{f+\tilde{f}}
=-2\int_{\R^{Nd}} \sum_{k=1}^{N/2} (\frac{\delta f}{f+\tilde{f}})_k \otimes (\frac{\delta f}{f+\tilde{f}})_k: a^k (f+\tilde{f})\\
+4\int_{\R^{Nd}} \sum_{k=1}^{N/2} \frac{\delta f}{f+\tilde{f}}\big(\frac{f}{f+\tilde{f}} [(\delta a^k\tilde{f})_{,k}+\delta b^k f]\big)_{,k}.
\end{multline*}
The first term is negative. Then, we can take the summation over $i$ and take the expectation.
Applying the H\"older inequality, one obtains
\[
\frac{d}{dt}E_1(t)
\le 4(E_1(t))^{1/2}
\left(\E \sum_{i=1}^N \int_{\R^{Nd}} \frac{1}{f+\tilde{f}} \left(\sum_{k=1}^{N/2} \Big(\frac{f}{f+\tilde{f}} [(\delta a^k\tilde{f})_{,k}+\delta b^k f]\Big)_{,k}\right)^2 \right)^{1/2}.
\]
We now estimate the term in the parenthesis.
By the coupling between $\theta$ and $\tilde{\theta}_i$, $\delta a^k$
is nonzero only for $k=i$ and $k=\tilde{\theta}_i(i)=:\tilde{i}$ (if $\tilde{\theta}_i(i)=\theta(i)$, then $\delta a^k=0$ for all $k$). Hence,
\begin{multline*}
\left|\sum_{k=1}^{N/2} \big(\frac{f}{f+\tilde{f}} [(\delta a^k\tilde{f})_{,k}+\delta b^k f]\big)_{,k} \right|
\le C\Big(1+|\tilde{f}_{ii}|+|f_i|+|\tilde{f}_i|+\frac{|f_i|^2+|\tilde{f}_i|^2}{f+\tilde{f}}\\
+|\tilde{f}_{\tilde{i}\tilde{i}}|+|f_{\tilde{i}}|+|\tilde{f}_{\tilde{i}}|
+\frac{|f_{\tilde{i}}|^2+|\tilde{f}_{\tilde{i}}|^2}{f+\tilde{f}} \Big).
\end{multline*}
Hence,
\begin{multline*}
\E \sum_{i=1}^{N} \int_{\R^{Nd}} \frac{1}{f+\tilde{f}} (\sum_{k=1}^{N/2} \big(\frac{f}{f+\tilde{f}} [(\delta a^k\tilde{f})_{,k}+\delta b^k f]\big)_{,k})^2 \\
\le C\E \sum_{i=1}^N \int \frac{1+|\tilde{f}_{ii}|^2+|f_i|^4+|\tilde{f}_i|^4}{f+\tilde{f}}
+C\sum_{i=1}^N \E \int \frac{1+|\tilde{f}_{\tilde{i}\tilde{i}}|^2+|f_{\tilde{i}}|^4+|\tilde{f}_{\tilde{i}}|^4}{f+\tilde{f}}
\end{multline*}

By Lemma \ref{lmm:onestepstability}, one finds that
\[
\int_{\R^{Nd}} \frac{|\tilde{f}_{\tilde{i}\tilde{i}}|^2}{f+\tilde{f}}(t)
\le C\left(\int_{\R^{Nd}} |(\log \tilde{f})_{\tilde{i}\tilde{i}}|^2\tilde{f}
+\int_{\R^{Nd}}  |(\log \tilde{f})_{\tilde{i}}|^4\tilde{f}  +\Delta t\right)|_{t=t_m}
\]
The point is that 
$\tilde{f}(t_m, \cdot)=f^{N}(t_m, \cdot)$ by \eqref{eq:initialvalueoffm} and thus is independent of $\theta$ and $\tilde{\theta}_i$. The random index $\tilde{i}$ is decoupled from $\tilde{f}$. Then, one has, for example,
\[
\E \int_{\R^{Nd}} |(\log \tilde{f})_{\tilde{i}\tilde{i}}|^2\tilde{f}(t_m)
\le C\frac{1}{N-1}\sum_{j=1}^N \int_{\R^{Nd}}  |\nabla_{v_j}^2\log f^N (t_m)|^2 f^N(t_m).
\]
 The other terms can be treated similarly. Then, applying Corollary \ref{cor:expectedparticledensity},  one obtains
\[
\frac{d}{dt}\sqrt{E_1(t)}\le C\sqrt{N} \Rightarrow E_1(t)\le CN (t-t_m)^2.
\]

For $E_2(t)$, we again fix $i$ and compute that
\begin{multline}\label{eq:E2derivative}
\frac{d}{dt}\int_{\R^{Nd}} \frac{|\delta f_{i}|^2}{f+\tilde{f}}
=-2\int_{\R^{Nd}} \sum_{k=1}^{N/2} (\frac{\delta f_i}{f+\tilde{f}})_k\cdot (\frac{\delta f_i}{f+\tilde{f}})_k:a^k (f+\tilde{f})\\
+2\int_{\R^{Nd}} \sum_{k=1}^{N/2} (\frac{\delta f_i}{f+\tilde{f}}) \cdot ( [a_i^k \delta f]_{,k}
+b_i^k \delta f+(\delta a^k \tilde{f})_{i,k}+(\delta b^k \tilde{f})_i )_{,k}\\
+2\int_{\R^{Nd}} \frac{\delta f_i}{f+\tilde{f}}\sum_{k=1}^{N/2} (\frac{\delta f_i}{f+\tilde{f}})_k
\cdot[-(\delta a^k \tilde{f})_{,k} -\delta b^k \tilde{f} ].
\end{multline}
For the third term, by Young's inequality, one has
\begin{multline*}
\int_{\R^{Nd}} \frac{\delta f_i}{f+\tilde{f}}\sum_{k=1}^{N/2} (\frac{\delta f_i}{f+\tilde{f}})_k
\cdot[-(\delta a^k \tilde{f})_{,k} -\delta b^k \tilde{f} ]\\
\le \epsilon \int_{\R^{Nd}} \sum_{k=1}^{N/2} |(\frac{\delta f_i}{f+\tilde{f}})_k|^2 (f+\tilde{f})
+C\int_{\R^{Nd}} \sum_{k=1}^{N/2} \frac{|\delta f_i|^2}{|f+\tilde{f}|^2} 
\frac{|(\delta a^k \tilde{f})_{,k} +\delta b^k \tilde{f}|^2}{f+\tilde{f}}.
\end{multline*}
The first term can be absorbed.  Taking expectation and the sum over $i$, applying the H\"older's inequality,
one has
\begin{multline*}
\sum_{i=1}^N \E \int_{\R^{Nd}} \sum_{k=1}^{N/2} \frac{|\delta f_i|^2}{|f+\tilde{f}|^2} 
\frac{|(\delta a^k \tilde{f})_{,k} +\delta b^k \tilde{f}|^2}{f+\tilde{f}} \le\\
 \left(\sum_{i=1}^N \E \int_{\R^{Nd}} \frac{|\delta f_i|^2}{f+\tilde{f}}\right)^{1/2}
\left(\sum_{i=1}^N \E \int_{\R^{Nd}} \left(\frac{|\delta f_i|}{|f+\tilde{f}|}
\sum_k \frac{|(\delta a^k \tilde{f})_{,k} +\delta b^k \tilde{f}|^2}{|f+\tilde{f}|^2}\right)^2(f+\tilde{f}) \right)^{1/2}.
\end{multline*}
The term in the second parenthesis can be controlled.
Here, we have split a portion of $|\delta f_i|$ into the second parenthesis, without using its smallness. This is fine since the third term on the right-hand side of \eqref{eq:E2derivative} is intuitively second order small, smaller than the main error term in the second term.

For the second term on the right-hand side of \eqref{eq:E2derivative}, the main term is
\begin{gather}\label{eq:localaux2}
\int_{\R^{Nd}} \sum_{k=1}^{N/2} (\frac{\delta f_i}{f+\tilde{f}}) \cdot (\delta a^k \tilde{f})_{i,k}
\end{gather}
Other terms can be estimated similarly as $E_1$.
For this term, since we do not have the third order derivative estimate, we perform integration by parts and 
obtain
\begin{multline*}
-\int_{\R^{Nd}} \sum_{k=1}^{N/2} (\frac{\delta f_i}{f+\tilde{f}})_k \cdot (\delta a^k \tilde{f})_{i}
\le \epsilon \int_{\R^{Nd}} \sum_{k=1}^{N/2} \left|(\frac{\delta f_i}{f+\tilde{f}})_k\right|^2(f+\tilde{f})
+C \int_{\R^{Nd}} \sum_{k=1}^{N/2} \frac{|(\delta a^k \tilde{f})_{i}|^2}{f+\tilde{f}}.
\end{multline*}
The first term is absorbed by the dissipation term. Taking the expectation and summation over $i$, the second term is controlled by 
$CN$. One finally obtains that
\[
\frac{d}{dt}E_2(t)\le C\sqrt{E_2}\sqrt{N}+CN.
\]
This implies that
\[
E_2(t)\le CN(t-t_m).
\]
\end{proof}

If we have an estimate for the third order derivatives similar as the second order derivative in \eqref{eq:integralsecondorder}, 
the term \eqref{eq:localaux2} may be controlled by $\sqrt{E_2}\sqrt{N}$. Using this, the local error bound could be improved
to $CN(t-t_m)^2$ which should be optimal. The estimate of the third order derivatives seems quite tricky and we would like to leave this for further study in the future.

\section{Closing the loose ends}\label{sec:looseends}

In this section, we provide the details and proofs that we skip in the previous sections. 
In particular, we provide the sketch of the proof of Lemma \ref{lmm:concentration}, and then fill in the details for the calculation of the mean field limit in the diffusion coefficient.

\subsection{The concentration estimate}\label{subsec:newlargedevi}

In this section, we outline the proof of Lemma \ref{lmm:concentration}. This is a consequence of the concentration of sub-Gaussians, and related to the large deviation estimate in \cite[Theorem 4]{jabinquantitative}.


\begin{proof}[Proof of Lemma \ref{lmm:concentration}]

Let $\{X_i\}_{i=1}^N$ be i.i.d. drawn from $\bar{f}(t,\cdot)$. By the assumption on $\psi$, the standard Hoeffding's inequality
\cite{vershynin2018high} gives that
\begin{gather*}
\P\left(\Big|\sum_{i=1}^N \psi(X_i)\Big| \ge y\right)\le 2\exp\left(-\frac{c_* y^2}{\sum_{i=1}^N \|\psi(X_i)\|_{\psi_2}}\right),
\quad \forall y\ge 0.
\end{gather*}

Let
\[
Y=\Big|\sum_{i=1}^N \psi(X_i)\Big|.
\]
Then, we have
\begin{multline*}
\E \exp\left(\frac{1}{N}Y^2\right)= \int_0^{\infty}\P(Y\ge y) \frac{2y}{N}\exp\left(\frac{y^2}{N}\right) dy\\
\le \int_0^{\infty}2\exp\left(-\frac{c_* y^2}{N\|\psi(X_1)\|_{\psi_2}}\right) \frac{2y}{N}\exp\left(\frac{y^2}{N}\right) dy.
\end{multline*}
Since $c_*/\|\psi(X_1)\|_{\psi_2}>1$, the integral then converges and the value is independent of $N$.
\end{proof}

\subsection{Mean field approximation of the diffusion coefficient}\label{subsec:meandiffusion}

In this subsection, we focus on the term involving the mean field approximation of the diffusion coefficient in the large particle number limit. In particular,
\begin{gather*}
\begin{split}
I_{22}&=\sum_{i=1}^N \int_{\R^{Nd}} |\nabla_{v_i}\log\bar{f}^N|^2 \Big\|\frac{1}{N-1}\sum_{j\neq i}A(v_i-v_j)-A(v_i; \bar{f}) \Big\|^2 f^N\\
&=N \int_{\R^{Nd}} |\nabla_{v_1}\log\bar{f}(v_1)|^2 \Big\|\frac{1}{N-1}\sum_{j\neq 1}A(v_1-v_j)-A(v_1; \bar{f}) \Big\|^2 f^N.
\end{split}
\end{gather*}
Our main goal is to prove the following.
\begin{proposition}\label{pro:I22}
Consider the quantity $I_{22}$. Suppose Assumptions \ref{ass:KA1}, \ref{ass:A2} and \ref{ass:initial} hold. Then, for fixed $T>0$,
there exists $\eta_*>0$ such that for all $\eta<\eta_*$, it holds that for all $t\le T$ that
\begin{gather}
I_{22}\le CH(f^N | \bar{f}^N)+C,
\end{gather}
where $C$ is independent of $N$ and $t$ (it may depend on $T$).
\end{proposition}

\begin{proof}

By Theorem~\ref{pro:gradlogdensity}, we find that if $A$ satisfies Assumptions \ref{ass:KA1} and \ref{ass:A2} and the initial density has the Gaussian tail as in Assumption \ref{ass:initial}, then
\begin{gather}\label{eq:sec6aux1}
|\nabla_{v_1}\log\bar{f}(v_1)| \le C(1+|v_1|).
\end{gather}
Then, our goal is thus to control
\begin{multline*}
\int_{\R^{Nd}} |\nabla_{v_1}\log\bar{f}(v_1)|^2 \Big\|\frac{1}{N-1}\sum_{j\neq 1}A(v_1-v_j)-A(v_1; \bar{f}) \Big\|^2 f^N \\
= \int_{\R^{Nd}} \Big\|\frac{1}{N-1}\sum_{j\neq 1}\psi(v_j; v_1) \Big\|^2 f^N=:A.
\end{multline*}
with
\begin{gather}
\psi(v_j; v_1):=|\nabla_{v_1}\log\bar{f}(v_1)|(A(v_1-v_j)-A(v_1; \bar{f})).
\end{gather}
By Lemma \ref{lmm:youngIneq}, one has
\begin{gather*}
A\le \frac{1}{(N-1)\eta}H(f^N |\bar{f}^N)
+\frac{1}{(N-1)\eta}\log \int_{\R^{Nd}} \exp\left(\frac{\eta}{N-1} \Big|\sum_{j\neq 1}\psi(v_j; v_1) \Big|^2\right)\bar{f}^N.
\end{gather*}

Next, we estimate $\psi(\cdot ; v_1)$. In particular, we show that it has a $\|\cdot\|_{\psi_2}$ norm that is uniform with respect to 
$v_1$. In fact, one has
\begin{gather}\label{eq:sec6aux2}
\sup_{v_1}|\psi(v_j; v_1)|\le C(1+|v_j|).
\end{gather}
To see this, we first split
\begin{gather*}
A(v_1-v_j)-A(v_1; \bar{f})=A(v_1-v_j)-A(v_1)-\left(\int_{\R^d} (A(v_1-v_j')-A(v_1))\bar{f}(v_j')dv_j'\right).
\end{gather*}
Using \eqref{eq:sec6aux1}, one finds that
\begin{multline*}
(1+|v_1|) |A(v_1-v_j)-A(v_1)|\\
\le \begin{cases}
C(1+|v_j|), & |v_j|\ge |v_1|/2,\\
C (1+|v_1|)\sup_{s\in [0, 1]}|\nabla A(v_1-s v_j)| |v_j|\le C|v_j|, & |v_j|<|v_1|/2.
\end{cases}
\end{multline*}
The second case follows from Assumption \ref{ass:A2}.
Hence, \eqref{eq:sec6aux2} holds.

By Lemma \ref{lmm:taildensity}, we can find $\eta_*>0$ such that for all $\eta<\eta_*$, one has
\begin{gather}
\sup_{t\le T}\|\sqrt{\eta}\psi(\cdot; v_1)\|_{\bar{f}(t), \psi_2}<c_*.
\end{gather}
An application of Lemma \ref{lmm:concentration} gives the result.
\end{proof}

\section{Conclusions}\label{sec:conclusion}

In this work, we have proposed a new particle system for Landau-type equations, inspired by a microscopic collision model, which maintains a strong connection to the underlying physics. This new particle system appears as a random batch system with interaction in the diffusion coefficient. Although our particle system is not merely a numerical method, it offers the advantage of computational complexity scaling as $O(N)$ when employed as such.

We have provided a quantitative analysis of the approximation of the particle system to Landau-type equations using the tool of relative entropy. The main technical challenge in our proof involves estimating the derivatives of logarithimic densities for the particle system, particularly ensuring that these estimates scale linearly with the particle number. The estimates for the first derivatives have been addressed by introducing a new advection term. A tail estimate for the Hessian of the logarithmic density of the mean field Landau like model is also needed for the law of large numbers in the diffusion coefficient. Our results indicate that as the frequency of collisions increases and the particle number grows, the system exhibits more chaos, thereby justifying the molecular chaos in a certain sense and the mean-field limit of the system.

For future work, we plan to explore the numerical properties of this new particle system when used as a numerical method for solving the actual Landau models. Moreover, investigating the approximation error of this new particle system when applied to true Landau models is an intriguing question. Furthermore, the collision model at random times could be considered. These studies may offer deeper insights into the molecular chaos of microscopic models for Landau equations.

\section*{Acknowledgement}

This work was financially supported by the National Key R\&D Program of China, Project Number 2021YFA1002800.
The work of K. Du was partially supported by the National Natural Science Foundation of China (12222103) and by the National Science and Technology Major Project (2022ZD0116401).
The work of L. Li was partially supported by NSFC 12371400 and 12031013,  Shanghai Municipal Science and Technology Major Project 2021SHZDZX0102, the Strategic Priority Research Program of Chinese Academy of Sciences, Grant No. XDA25010403.

\appendix

\section{Some detailed computation}\label{app:missingdetails}

The objective function can be written as
\begin{gather}
u^k= e^h \sum_{\alpha\in \mathcal{I}_k} \varphi(h_{\alpha})
\end{gather}
where $\varphi(x)=|x|^p$.

Direct computation gives 
\begin{multline*}
u_t^k-\sum_{i=1}^{N/2}((\tilde{a}^i u^k )_{,ii}+(\tilde{b}^i u^k )_{,i})
=u_t^k-\sum_{i=1}^{N/2}(a^i:u^k_{ii}+b^i\cdot u^k_i+c^iu^k )\\
=(\sum_{\alpha\in \cI_k}\varphi(h_{\alpha}))(\partial_t(e^h)-\sum_i a^i:(e^h)_{ii}-\sum_i b^i\cdot(e^h)_i-\sum_i c^i e^h)\\
+\sum_{\alpha\in \cI_k} \nabla\varphi(h_{\alpha})\cdot (\partial_t h_{\alpha}-2\sum_i a^i: h_i\otimes h_{\alpha i}
-\sum_i a^i h_{\alpha ii}-\sum_i b^i\cdot h_{\alpha i})\\
-\sum_{\alpha\in \cI_k} \nabla^2\varphi(h_{\alpha}):\sum_i a^i: h_{i\alpha}\otimes h_{i\alpha}.
\end{multline*}
It is clear that 
\[
\partial_t(e^h)-\sum_i a^i:(e^h)_{ii}-\sum_i b^i\cdot(e^h)_i-\sum_i c^i e^h=0
\]
by the equation of $h$. In fact, this is just the Fokker-Planck equation for $e^h$.
Moreover, 
\begin{multline*}
\partial_t h_{\alpha}-2\sum_i a^i: h_i\otimes h_{\alpha i}
-\sum_i a^i h_{\alpha ii}-\sum_i b^i\cdot h_{\alpha i} \\
=(\partial_t h-\sum_i a^i:(h_{ii}+h_i\otimes h_i)-\sum_i b^i h_i-\sum_i c^i)_{\alpha}\\
+a^k_{\alpha}:(h_{kk}+h_k\otimes h_k)+b^k_{\alpha}\cdot h_k+c^k_{\alpha}.
\end{multline*}
Hence, one has
\begin{multline}\label{eq:appderivative1}
u_t^k-\sum_{i=1}^{N/2}((\tilde{a}^i u^k )_{,ii}+(\tilde{b}^i u^k )_{,i})
=
\sum_{\alpha\in \cI_k} \nabla\varphi(h_{\alpha})\cdot (a^k_{\alpha}:(h_{kk}+h_k\otimes h_k)+b^k_{\alpha}\cdot h_k+c^k_{\alpha})\\
-\sum_{\alpha\in \cI_k} \nabla^2\varphi(h_{\alpha}):\sum_i a^i: h_{i\alpha}\otimes h_{i\alpha}.
\end{multline}
With the concrete expression of $\varphi$, \eqref{eq:eqnforu1} is obtained.

With \eqref{eq:auxadvection}, one finds that
\begin{gather*}
\begin{split}
& u^k_t-\gamma \sum_{\alpha\in \cI_k} (e^h h_{\alpha} |h_{\alpha}|^p)_{,\alpha}-\sum_i (b^i u^k)_i-\sum_i(a^i u^k)_{,ii}\\
& \le -\kappa p(p-1) e^h\sum_{\alpha\in \cI_k} |h_{\alpha}|^{p-2}\sum_i |h_{\alpha i}|^2
-\gamma e^h \sum_{\alpha\in \cI_k} |h_{\alpha}|^{p+2}+\gamma C e^h\sum_{\beta\in \cI_k} |h_{\alpha}|^p |h_{\alpha\alpha}|\\
&+C p e^h(\sum_{\alpha\in \cI_k} |h_{\alpha }|^{p-1})\sum_{\beta\in \cI_k}(|h_{\beta\beta}|+|h_{\beta}|^2+|h_{\beta}|+1),
\end{split}
\end{gather*}
where $C$ is the norm given by the coefficients like $a_{\alpha}^k$, which is independent of $N$.

Using the Young's inequality, one has the controls like
\[
\begin{split}
& \gamma e^h |h_{\alpha}|^p |h_{\alpha\alpha}|
\le \frac{1}{2}\gamma e^h|h_{\alpha}|^{p+2}
+\frac{1}{2}\gamma e^h  |h_{\alpha}|^{p-2}|h_{\alpha\alpha}|^2,\\
& pe^h (\sum_{\alpha\in \cI_k} |h_{\alpha}|^{p-1})(\sum_{\beta\in \cI_k} |h_{\beta}|^2)\le 
2 pe^h (\sum_{\alpha\in \cI_k} |h_{\alpha}|^{p+1})\le \frac{\gamma}{4}e^h\sum_{\alpha}|h_{\alpha}|^{p+2}
+C(\gamma, p) e^h\sum_{\alpha}|h_{\alpha}|^{p}.\\
\end{split}
\]
Then, \eqref{eq:keyeqn1}  is obtained.

\bibliographystyle{plain}
\bibliography{landau}

\end{document}